\documentclass[12pt]{article}
\usepackage{mystyle}
\begin{document}
\title{The Ultrapower Axiom and the equivalence between strong compactness and supercompactness}
\author{Gabriel Goldberg}
\maketitle
\begin{abstract}
The relationship between the large cardinal notions of strong compactness and supercompactness cannot be determined under the standard ZFC axioms of set theory. Under a hypothesis called the Ultrapower Axiom, we prove that the notions are equivalent except for a class of counterexamples identified by Menas. This is evidence that strongly compact and supercompact cardinals are equiconsistent.
\end{abstract}
\section{Introduction}
{\it How large is the least strongly compact cardinal?} Keisler-Tarski \cite{Tarski} asked whether it must be larger than the least measurable cardinal, and Solovay later conjectured that it is much larger: in fact, he conjectured that every strongly compact cardinal is supercompact. His conjecture was refuted by Menas \cite{Menas}, who showed that the least strongly compact limit of strongly compact cardinals is not supercompact. But Tarski's question was left unresolved until in a remarkable pair of independence results, Magidor showed that the size of the least strongly compact cannot be determined using only the standard axioms of set theory (ZFC). More precisely, it is consistent with ZFC that the least strongly compact cardinal is the least measurable cardinal, but it is also consistent that the least strongly compact cardinal is the least supercompact cardinal. 

One of the most prominent open questions in set theory asks whether a weak variant of Solovay's conjecture might still be true: {\it is the existence of a strongly compact cardinal equiconsistent with the existence of a supercompact cardinal?} Since inner model theory is essentially the only known way of proving nontrivial consistency strength lower bounds, answering this question probably requires generalizing inner model theory to the level of strongly compact and supercompact cardinals. In this paper, we show that in any canonical inner model built by anything like today's inner model theoretic methodology, the least strongly compact cardinal is supercompact. This suggests that strong compactness and supercompactness are equiconsistent.

The precise statement of our theorem involves a combinatorial principle called the Ultrapower Axiom (UA). This principle holds in all known inner models as a direct consequence of the current methodology of inner model theory. It is therefore expected to hold in any future inner model. We prove:

\begin{thm}[UA]
The least strongly compact cardinal is supercompact.
\end{thm}
 
The proof of this fact takes up the first half of this paper.

Given the supercompactness of the first strongly compact cardinal under UA, it is natural to wonder about the second one. The proof that the first strongly compact cardinal is supercompact {\it does not} generalize to show that the second strongly compact is supercompact, at least not in any obvious way. The second half of the paper is devoted to proving  that assuming UA, every strongly compact cardinal is supercompact with the possible exception of limits:
\begin{thm}[UA]
If \(\kappa\) is a strongly compact cardinal, either \(\kappa\) is supercompact or \(\kappa\) is a limit of supercompact cardinals.
\end{thm}

In other words, the Ultrapower Axiom implies that the only counterexamples to Solovay's conjecture are the ones discovered by Menas. Given the amount of structure that appears in the course of our proof, without even a hint of inconsistency, it seems likely that the case of strongly compact limits of supercompact cardinals really can occur. In other words, we believe the Ultrapower Axiom can hold simultaneously with the existence of Menas's counterexamples. This would mean that UA is consistent with the existence of a strongly compact limit of supercompact cardinals. The only known way of proving the consistency of UA with large cardinals is by building canonical inner models\footnote{Current forcing techniques cannot build a model of UA with even one measurable cardinal. Forcing UA in the presence of a supercompact cardinal seems like a much harder problem.}, so the only conceivable explanation of the consistency of UA with a strongly compact limit of supercompact cardinals would be that there is a canonical inner model with many supercompact cardinals. We therefore think the results of this paper provide compelling evidence that such an inner model exists.

This paper is in a sense the sequel to \cite{Frechet} and we will need to cite many of the results of that paper. This paper can be understood without reading \cite{Frechet}, however, if one is willing to take the results of that paper on faith.
\section{Preliminaries}
\subsection{Uniform ultrafilters}
In this section we define two notions of ``uniform ultrafilter" and describe how they are related.
\begin{defn}
If \(\alpha\) is an ordinal, the {\it tail filter} on \(\alpha\) is the filter generated by sets of the form \(\alpha\setminus \beta\) where \(\beta < \alpha\). An ultrafilter on \(\alpha\) is {\it tail uniform} (or just {\it uniform}) if it extends the tail filter on \(\alpha\).
\end{defn}

Thus an ultrafilter \(U\) on an ordinal is uniform if and only if all sets in \(U\) have the same supremum. 

Notice that if \(\alpha\) is an ordinal, there is a {\it uniform principal ultrafilter} on \(\alpha+1\):
\begin{defn}
The {\it uniform principal ultrafilter} on \(\alpha+1\) is the ultrafilter \(P_\alpha = \{X\subseteq \alpha+1 : \alpha\in X\}\).
\end{defn}

Every ultrafilter on an ordinal restricts to a uniform ultrafilter on an ordinal in the following sense:
\begin{defn}
If \(U\) is an ultrafilter and \(X\in U\), then the {\it restriction of \(U\) to \(X\)} is the ultrafilter \(U\restriction X = U\cap P(X)\).
\end{defn}

\begin{lma}
If \(U\) is an ultrafilter on an ordinal, then there is a unique ordinal \(\alpha\) such that \(\alpha\in U\) and \(U\restriction \alpha\) is uniform.\qed
\end{lma}

\begin{defn}
If \(U\) is a uniform ultrafilter, the {\it space} of \(U\), denoted \(\textsc{sp}(U)\), is the unique ordinal \(\alpha\) such that \(\alpha\in U\).
\end{defn}

A somewhat different notion of uniformity is also in use: 

\begin{defn}
If \(X\) is an infinite set, the {\it Fr\'echet filter} on \(X\) is the filter generated by sets of the form \(X\setminus S\) where \(|S| < |X|\). An ultrafilter on \(X\) is {\it Fr\'echet uniform} if it extends the Fr\'echet filter on \(X\).
\end{defn}

Thus an ultrafilter \(U\) is Fr\'echet uniform if and only if all sets in \(U\) have the same cardinality.

Arbitrary ultrafilters restrict to Fr\'echet uniform ultrafilters:
\begin{lma}
If \(U\) is an ultrafilter, there is some \(X\in U\) such that \(U\restriction X\) is Fr\'echet uniform.\qed
\end{lma}

The following fact explains the relationship between tail uniform and Fr\'echet uniform ultrafilters:

\begin{lma}\label{RegTailFrec}
If \(\delta\) is a regular cardinal, then an ultrafilter on \(\delta\) is uniform if and only if it is Fr\'echet uniform.\qed
\end{lma} 

Since we will mostly be concerned with ultrafilters on regular cardinals in this paper, the distinction will not be that important here.

\begin{lma}\label{CofUniform}
Suppose \(\alpha\) is an ordinal and \(f : \alpha\to \alpha'\) is a weakly increasing cofinal function. Then for any uniform ultrafilter \(U\) on \(\delta\), \(f_*(U)\) is a uniform ultrafilter on \(\alpha'\).\qed
\end{lma}

\begin{defn}
An ordinal \(\alpha\) is {\it tail uniform} (or just {\it uniform}) if it carries a tail uniform countably complete ultrafilter. A cardinal \(\lambda\) is {\it Fr\'echet uniform} (or just {\it Fr\'echet}) if it carries a Fr\'echet uniform countably complete ultrafilter.
\end{defn}

The following is a consequence of \cref{RegTailFrec}:

\begin{prp}
If \(\delta\) is a regular cardinal, then \(\delta\) is tail uniform if and only if \(\delta\) is Fr\'echet.\qed
\end{prp}

If \(\delta\) is a regular cardinal, we say \(\delta\) is {\it uniform} to mean that \(\delta\) is tail uniform or equivalently that \(\delta\) is Fr\'echet.

The following is a consequence of \cref{CofUniform}:

\begin{prp}\label{CofUniform2}
An ordinal \(\alpha\) is tail uniform if and only if \(\textnormal{cf}(\alpha)\) is uniform.\qed
\end{prp}

There are characterizations of uniform ordinals that are closely analogous to Scott's elementary embedding characterization of measurable cardinals.

\begin{prp}\label{TailUnifChar}
An ordinal \(\alpha\) is tail uniform if and only if there is an elementary embedding \(j :V\to M\) that is discontinuous at \(\alpha\).\qed
\end{prp}

\begin{prp}\label{FrechUnifChar}
A cardinal \(\lambda\) is Fr\'echet if and only if there exist elementary embeddings \(V\stackrel{j}\longrightarrow M\stackrel{k}{\longrightarrow}N\) such that \(\sup j[\lambda]\leq \textsc{crt}(k) < j(\lambda)\).\qed
\end{prp}
\subsection{Large cardinals}
\begin{defn}
Suppose \(\kappa\leq \lambda\) are cardinals.

We say \(\kappa\) is {\it \(\lambda\)-supercompact} if there is an elementary embedding with critical point \(\kappa\) from the universe of sets into an inner model that is closed under \(\lambda\)-sequences. 
 
We say \(\kappa\) is {\it \(\lambda\)-strongly compact} if there is an elementary embedding \(j\) with critical point \(\kappa\) from the universe of sets into an inner model \(M\) such that every set \(A\subseteq M\) with \(|A|\leq \lambda\) is contained in a set \(A'\in M\) with \(|A'|^M < j(\kappa)\).
\end{defn}

We state the equivalence between the embedding formulations of these large cardinal axioms and the ultrafilter versions.

\begin{defn}
Suppose \(A\) is a set and \(\mathcal U\) is an ultrafilter on \(X\subseteq P(A)\). We say \(\mathcal U\) is {\it fine} if any \(a\in A\) belongs to \(U\)-almost all \(\sigma\in X\). We say \(\mathcal U\) is {\it normal} if any choice function on \(X\) is constant on a \(\mathcal U\)-large set.
\end{defn}

The following theorems appear in \cite{Kanamori} except for \cref{StrongcompactnessEquiv} (4) which is due to Ketonen \cite{Ketonen}:

\begin{thm}\label{SupercompactnessEquiv}
Suppose \(\lambda\) is a cardinal. Then the following are equivalent:
\begin{enumerate}[(1)]
\item \(\kappa\) is \(\lambda\)-supercompact.
\item There is a normal fine \(\kappa\)-complete ultrafilter on \(P_\kappa(\lambda)\).
\item There is a normal fine ultrafilter on \(P(\lambda)\) with completeness \(\kappa\).\qed
\end{enumerate}
\end{thm}

\begin{thm}\label{StrongcompactnessEquiv}
Suppose \(\lambda\) is a cardinal. Then the following are equivalent:
\begin{enumerate}[(1)]
\item \(\kappa\) is \(\lambda\)-strongly compact.
\item There is a \(\kappa\)-complete fine ultrafilter on \(P_\kappa(\lambda)\).
\item Every \(\kappa\)-complete filter generated by at most \(\lambda\) sets extends to a \(\kappa\)-complete ultrafilter. 
\item Every regular cardinal \(\delta\) such that \(\kappa\leq \delta\leq\lambda\) carries a \(\kappa\)-complete uniform ultrafilter.\qed
\end{enumerate}
\end{thm}
\subsection{Comparisons and the seed order}
\begin{defn}
We say \(j\) is an {\it ultrapower embedding} of \(M\) if \(N = H^N(j[M]\cup \{x\})\) for some \(x\in N\). We say \(j\) is an {\it internal ultrapower embedding} of \(M\) if \(j\) is an ultrapower embedding of \(M\) and \(j\restriction x\in M\) for all \(x\in M\).
\end{defn}

\begin{defn}
If \(P\) is an inner model and \(U\) is a \(P\)-ultrafilter, we denote by \(j_U^P : P\to M_U^P\) the ultrapower of \(P\) by \(U\) using functions in \(P\). For any function \(f\in P\) defined on a set in \(U\), \([f]^P_U\) denotes the point in \(M_U^P\) represented by \(f\).
\end{defn}

Given this terminology, an embedding \(j :M \to N\) is an ultrapower embedding if and only if there is an \(M\)-ultrafilter \(U\) such that \(j = j_U^M\), and an internal ultrapower embedding if and only if \(U\) can be chosen to belong to \(M\). (In this case, any \(M\)-ultrafilter \(U\) such that \(j = j_U^M\) will belong to \(M\).)

\begin{defn}
Suppose \(M_0,M_1,\) and \(N\) are transitive models. We write \[(i_0,i_1) : (M_0,M_1)\to N\] to denote that \(i_0 : M_0\to N\) and \(i_1 : M_1\to N\) are elementary embeddings.
\end{defn}

\begin{defn}
Suppose \(j_0 : V\to M_0\) and \(j_1 : V\to M_1\) are ultrapower embeddings. A pair of internal ultrapower embeddings \((i_0,i_1) : (M_0,M_1)\to N\) is a {\it comparison} of \((j_0,j_1)\) if \(i_0\circ j_0 = i_1 \circ j_1\).
\end{defn}

The notion of a comparison leads to the definitions of the Ultrapower Axiom and the seed order.

\begin{ua}
Every pair of ultrapower embeddings has a comparison.
\end{ua}

\begin{defn}
The {\it seed order} is the relation defined on \(U_0,U_1\in \Un\) by setting \(U_0\swo U_1\) if there is a comparison \((i_0,i_1) : (M_{U_0},M_{U_1})\to N\) of \((j_{U_0},j_{U_1})\) such that \(i_0([\text{id}]_{U_0}) < i_1([\text{id}]_{U_1})\).
\end{defn}

\begin{thm}[UA]
The seed order is a wellorder of \(\Un\).
\end{thm}

We discuss some results from \cite{SO} that single out for any pair of ultrapowers a unique ``optimal" comparison which we call a {\it canonical comparison}.

\begin{defn}
Suppose \(j_0 : V\to M_0\) and \(j_1 : V\to M_1\) are ultrapower embeddings. A comparison \((i_0,i_1) :(M_0,M_1)\to N\) of \((j_0,j_1)\) is
\begin{enumerate}[(1)]
\item {\it minimal} if \(N = H^N(i_0[M_0]\cup i_1[M_1])\)
\item {\it canonical} if for any comparison \((i_0',i_1') : (M_0,M_1)\to N'\), there is an elementary embedding \(h : N \to N'\) such that \(h\circ i_0 = i_0'\) and \(h\circ i_1 = i_1'\).
\item a {\it pushout} if for any comparison \((i_0',i_1') : (M_0,M_1)\to N'\), there is an internal ultrapower embedding \(h : N \to N'\) such that \(h\circ i_0 = i_0'\) and \(h\circ i_1 = i_1'\).
\end{enumerate}
\end{defn}

Under UA, these notions all coincide. The following facts come from \cite{RF} Section 5.

\begin{lma}
A comparison of a pair of ultrapower embeddings is canonical if and only if it is their unique minimal comparison.\qed
\end{lma}

\begin{thm}[UA]\label{Pushout}
Every pair of ultrapower embeddings has a pushout.\qed
\end{thm}

It follows that every pair of ultrapower embeddings has a unique pushout, a unique canonical comparison, and a unique minimal comparison, and these all coincide. For simplicity, in the context of UA, we will refer to this comparison as {\it the canonical comparison}, and largely avoid using the other terms. We will make use of the following fact:

\begin{thm}\label{CanonicalInternal}
Suppose \(j_0 : V\to M_0\) and \(j_1 : V\to M_1\) are ultrapower embeddings and \((i_0,i_1) : (M_0,M_1)\to N\) is a pushout of \((j_0,j_1)\). Suppose \(k : N\to N'\) is an ultrapower embedding of \(N\). Then the following are equivalent:
\begin{enumerate}[(1)]
\item \(k\) is an internal ultrapower embedding of \(N\)
\item \(k\) is amenable to \(M_0\) and \(M_1\).\qed
\end{enumerate}
\end{thm}

\subsection{Limits and the Ketonen order}
We need a slight generalization of the notion of a limit of ultrafilters defined in \cite{SO}. 

\begin{defn}
Suppose \(U\) is a countably complete ultrafilter and \(Z\) is an \(M_U\)-ultrafilter on an ordinal \(\beta\). Then the {\it \(U\)-limit of \(Z\)} is the ultrafilter \[U^-(Z) = \{X\subseteq \alpha : j_U(X)\cap \beta\in Z\}\] where \(\alpha\) is least such that \(j_U(\alpha) \geq \beta\).
\end{defn}

The main difference here is that we do not assume \(Z\in M_U\). For the most part, however, we will only be concerned with the case \(Z\in M_U\).

\begin{defn}
The {\it Ketonen order} is defined on \(U,W\in\Un\) by setting \(U\sE W\) if there is some \(Z\in \Un_{\leq [\text{id}]_W}^{M_W}\) such that \(U = W^-(Z)\).
\end{defn}

The following fact restates the two main theorems of \cite{SO}.

\begin{lma}
The Ketonen order extends the seed order. The two orders coincide if and only if the Ultrapower Axiom holds.
\end{lma}

We will make use of the notion of the {\it translation function} \(t_U\) associated to a countably complete ultrafilter \(U\) under UA:

\begin{defn}[UA]
Suppose \(U\) and \(W\) are uniform countably complete ultrafilters. Then \(\tr U W\) is the \(\sE^{M_U}\)-least \(Z\in \Un^{M_U}\) such that \(U^-(Z) = W\).
\end{defn}

The following trivial bound is worth noting:
\begin{lma}[UA]\label{BoundingLemma}
For any \(U,W\in \Un\), \(\tr U W \E^{M_U} j_U(W)\).
\begin{proof}
Note that \(U^-(j_U(W)) = W\).
\end{proof}
\end{lma}

This bound is interesting given the following theorem, \cite{SO} Lemma 5.6, which relates translation functions to comparisons:
\begin{thm}[UA]\label{Reciprocity1}
Suppose \(U_0,U_1\in \Un\). Let \(M_0 = M_{U_0}\) and \(M_1 = M_{U_1}\). Let \(i_0 = j^{M_0}_{\tr {U_0} {U_1}}\) and \(i_1 = j^{M_1}_{\tr {U_1} {U_0}}\). Then \((i_0,i_1) : (M_{U_0},M_{U_1})\to N\) is the canonical comparison of \((j_{U_0},j_{U_1})\).\qed
\end{thm}

We also use the following fact from \cite{SO} Proposition 5.8:
\begin{prp}[UA]\label{OrderPreserving}
For any \(U\in \Un\), the function \(t_U : (\Un,\sE)\to (\Un,\sE)^{M_U}\) is order preserving. 
\end{prp}

We will need the following fact:
\begin{lma}[UA]\label{kInternal}
Suppose \(D,U\in \Un\) and \(k : M_D\to M_U\) is an elementary embedding such that \(k\circ j_D = j_U\). Then the following are equivalent:
\begin{enumerate}[(1)]
\item \(k\) is an internal ultrapower embedding of \(M_D\).
\item \(\tr U D = P_{k([\textnormal{id}]_D)}^{M_U}\).
\item \(\tr U D\in k[M_D]\).
\end{enumerate}
\begin{proof}
(1) implies (2) by \cref{Reciprocity1} since assuming (1), \((k,\text{id})\) is a comparison of \((j_D,j_U)\) and \(P_{k([\textnormal{id}]_D)}^{M_U}\) is the \(M_U\)-ultrafilter derived from \(\text{id}\) using \(k([\text{id}]_D)\). (2) also easily implies (1). 

(2) implies (3) trivially. We finally show that (3) implies (2). Fix \(Z\in M_D\) such that \(k(Z) = \tr U D\). It is easy to see that \(Z = \tr D D\). We claim \(Z = P^{M_D}_{[\text{id}]_D}\). Otherwise, since \(D^-(P^{M_D}_{[\text{id}]_D}) = D\), we have \(Z\sE P^{M_D}_{[\text{id}]_D}\) in \(M_D\), or in ther words \(Z\in \Un_{\leq [\text{id}]_D}^{M_D}\). But the existence of such a \(Z\) implies \(D\sE D\), which contradicts that \(\sE\) is a strict partial order. Hence \(Z = P^{M_D}_{[\text{id}]_D}\). It follows that \(\tr U D = k(Z) = P_{k([\textnormal{id}]_D)}^{M_U}\).
\end{proof}
\end{lma}

\section{Ketonen embeddings}
In order to prove the supercompactness of the least strongly compact cardinal, one must somehow define elementary embeddings witnessing its supercompactness. Of course this cannot be done in ZFC alone, and at first it is hard to see how UA will help. In this section, we define ultrafilters \(U_\delta\) that will give rise to embeddings roughly witnessing \(\delta\)-supercompactness. It turns out that these ultrafilters {\it were} in a sense first isolated in the ZFC context by Ketonen \cite{Ketonen}. We begin by expositing some of Ketonen's work. 

\subsection{Ketonen embeddings}
Here we use the wellfoundedness of the Ketonen order to identify certain minimal ultrafilters which we call {\it Ketonen ultrafilters}. These ultrafilters will witness (approximately) the \(\delta\)-supercompactness of the least cardinal \(\kappa\) that is strongly compact to \(\delta\).
\begin{defn}
Suppose \(\delta\) is a regular uncountable cardinal. An elementary embedding \(j :  V\to M\) is a {\it Ketonen embedding} at \(\delta\) if the following hold where \(\delta_* = \sup j[\delta]\):
\begin{enumerate}[(1)]
\item \(M = H^{M}(j[V]\cup \{\delta_*\})\)
\item \(\delta_*\) is not tail uniform in \(M\).
\end{enumerate}
A uniform ultrafilter \(U\) is {\it Ketonen} at \(\delta\) if it is the ultrafilter derived from a Ketonen embedding \(j\) at \(\delta\) using \(\sup j[\delta]\).
\end{defn}

Thus the ultrapower embedding associated to a Ketonen ultrafilter is a Ketonen embedding. It is easy to see that if \(U\) is a Ketonen ultrafilter and \(\delta\) is uniform, then \(U\) is a uniform ultrafilter on \(\delta\). Otherwise \(U = P_\delta\) is a uniform ultrafilter on \(P_{\delta+1}\).

There is a simple combinatorial characterization of Ketonen ultrafilters:
\begin{lma}\label{KetonenCombinatorial}
A countably complete ultrafilter \(U\) is Ketonen at the regular cardinal \(\delta\) if and only if the following hold:
\begin{enumerate}[(1)]
\item Every regressive function on \(\textsc{sp}(U)\) is bounded on a \(U\)-large set.
\item \(\{\alpha < \textsc{sp}(U): \alpha\text{ is not tail uniform}\}\in U\).\qed
\end{enumerate}
\end{lma}
Because it is convenient, we allow that when \(\delta\) is a regular cardinal that is not tail uniform, the identity is Ketonen at \(\delta\), and therefore the uniform principal ultrafilter on \(\delta+ 1\) is Ketonen at \(\delta\).
\begin{thm}[Ketonen]\label{KetonenMinimal}
Suppose \(\delta\) is a regular cardinal and \(U\) is a uniform countably complete ultrafilter. Then the following are equivalent:
\begin{enumerate}[(1)]
\item \(U\) is Ketonen at \(\delta\).
\item \(U\) is an \(\sE\)-minimal element of \(\Un\setminus \Un_{<\delta}\).
\end{enumerate}
\begin{proof}
We first show (1) implies (2). Assume (1) and fix \(W\sE U\). Let \(\delta_* = [\text{id}]_U = \sup j_U[\delta]\). Then there is some \(W'\in \Un^{M_U}_{\leq\delta_*}\) such that \(W = U^-(W')\). Since \(U\) is zero order, \(\Un^{M_U}_{\delta_*} = \emptyset\), so \(W'\in \Un^{M_U}_{<\delta_*}\). It follows easily that \(W \in \Un_{<\delta}\). Since any \(W\sE U\) belongs to \(\Un_{<\delta}\), (2) holds.

We now show (2) implies (1). Assume (2). Let \(\delta_* = \sup j_U[\delta]\). Assume towards a contradiction that in \(M_U\), there is a uniform countable complete ultrafilter \(W'\) such that \(\delta_* \leq \textsc{sp}(W') \leq [\text{id}]_U\). Let \(\alpha\) be least such that \(j_U(\alpha) \geq \textsc{sp}(W')\), so \(\alpha\geq \delta\). Let \(W = U^-(W')\). Then \(W \sE U\). This contradicts that \(U\) is \(\sE\)-minimal. It follows that in \(M_U\), there is no uniform countable complete ultrafilter \(W'\) such that \(\delta_* \leq \textsc{sp}(W') \leq [\text{id}]_U\). It follows in particular that \(\Un_{\delta_*}^{M_U}= \emptyset\). Since the principal ultrafilter \(P_{\delta_*}\) is a uniform countably complete ultrafilter on \(\delta_*+1\), it follows that \([\text{id}]_U < \delta_* + 1\), or in other words \(\delta_* = [\text{id}]_U\).  That is, \(U\) is Ketonen.
\end{proof}
\end{thm}

\begin{cor}\label{KetonenExistence}
For every regular cardinal \(\delta\), there is a Ketonen ultrafilter at \(\delta\).
\begin{proof}
This is immediate from the wellfoundedness of the seed order.
\end{proof}
\end{cor}

We end this section with Ketonen's remarkable theorem on the covering properties of Ketonen ultrapowers.

\begin{defn}
Suppose \(M\) is an inner model, \(\lambda\) is a cardinal, and \(\lambda'\) is an \(M\)-cardinal. We say \(M\) has the {\it \((\lambda,\lambda')\)-covering property} if every set \(A\subseteq M\) with \(|A|\leq \lambda\) is contained in a set \(A'\in M\) with \(|A'|^{M} \leq \lambda'\). We say \(M\) has the {\it \((\lambda,{<}\lambda')\)-covering property}  if every set \(A\subseteq M\) with \(|A|\leq \lambda\) is contained in a set \(A'\in M\) with \(|A'|^{M} < \lambda'\). 
\end{defn}

\begin{lma}\label{EACov}
Suppose \(j : V\to M\) is an elementary embedding. Suppose \(\lambda\) is a cardinal and \(\lambda'\) is an \(M\)-cardinal. Assume there exist sets \(X\) and \(X'\) such that \(|X| = \lambda\), \(X'\in M\), \(|X'|^M = \lambda'\), and \(j[X]\subseteq X'\). Then for any set \(Y\) with \(|Y| \leq \lambda\), \(j[Y]\) is contained in a set \(Y'\in M\) with \(|Y'|^M \leq \lambda'\).
\begin{proof}
Fix a surjection \(f : X\to Y\). Let \(Y' = j(f)[X']\). Then \(Y'\in M\) and since \(j(f)\) is a surjection from \(X'\) to \(Y'\) in \(M\), \(|Y'|^M \leq |X'|^M = \lambda'\). Finally \(j[Y] = j[f[X]] = j(f)[j[X]]\subseteq j(f)[X'] = Y'\). 
\end{proof}
\end{lma}

The following lemma is a standard property of ultrapower embeddings (and in fact it holds in a bit more generality).

\begin{lma}\label{UltCov}
Suppose \(j : V\to M\) is an ultrapower embedding. Suppose \(\lambda\) is a cardinal and \(\lambda'\) is an \(M\)-cardinal. Assume there exist sets \(X\) and \(X'\) such that \(|X| = \lambda\), \(X'\in M\), \(|X'|^M = \lambda'\), and \(j[X]\subseteq X'\).  Then \(M\) has the \((\lambda,\lambda')\)-covering property.
\begin{proof}
Fix a set \(A\subseteq M\) with \(|A|\leq \lambda\). Fix \(e\in M\) such that \(M = H^M(j[V]\cup \{e\})\). Fix a set of functions \(Y\) such that \(|Y| = \lambda\) and \(A = \{j(f)(e) : f\in Y\}\). By \cref{EACov}, let \(Y'\in M\) be a set of functions with \(|Y'|^M \leq \lambda'\) and \(j[Y]\subseteq Y'\). Then \(A\subseteq \{g(e) : g\in Y'\wedge e\in \text{dom}(g)\}\). Therefore let \(A' =  \{g(e) : g\in Y'\wedge e\in \text{dom}(g)\}\). Then \(A'\in M\) contains \(A\) and \(|A'|^M \leq |A'|^M \leq\lambda'\). This proves the lemma.
\end{proof}
\end{lma}

\begin{cor}[Ketonen]\label{KetonenCov}
Suppose \(j : V\to M\) is an ultrapower embedding and \(\delta\) is a regular cardinal. Let \(\delta' = \textnormal{cf}^M(\sup j[\delta])\). Then \(M\) has the \((\delta,\delta')\)-covering property.
\begin{proof}
It suffices by \cref{UltCov} to show that there exist sets \(X\) and \(X'\) such that \(|X| = \delta\), \(X'\in M\), \(|X'|^M = \delta'\), and \(j[X]\subseteq X'\). Let \(X'\in M\) be a closed unbounded subset of \(\sup j[\delta]\) such that the ordertype of \(X'\) is \(\delta'\). Let \(X = j^{-1}[X']\). Then since \(j\) is continuous at ordinals of cofinality \(\omega\), \(j^{-1}[X']\) is \(\omega\)-closed unbounded in \(\delta\). Since \(\delta\) is regular, \(|X| = \delta\). Therefore \(X\) and \(X'\) are as desired.
\end{proof}
\end{cor}

\begin{thm}[Ketonen]\label{KetonenRegularity}
Suppose \(\delta\) is a regular cardinal and \(j : V\to M\) is a Ketonen embedding at \(\delta\). Suppose \(\gamma \leq \delta\) is such that every regular cardinal in the interval \([\gamma,\delta]\) is uniform. Then \(M\) has the \((\delta,{<}j(\gamma))\)-covering property.
\begin{proof}
Let \(\delta_* = \sup j[\delta]\). Let \(\delta' = \text{cf}^M(\delta_*)\).  By \cref{KetonenCov}, it suffices to show that \(\delta' < j(\gamma)\). By the elementarity of \(j\), every \(M\)-regular cardinal in the interval \([j(\gamma), j(\delta)]\) is Fr\'echet. On the other hand, by the definition of a Ketonen embedding, \(\delta_*\) is not tail uniform in \(M\), and therefore by \cref{CofUniform2}, \(\delta'\) is not Fr\'echet in \(M\). Since \(\delta'\) is an \(M\)-regular cardinal, it follows that \(\delta'\notin [j(\gamma),j(\delta)]\). Since \(\delta' \leq j(\delta)\), it follows that \(\delta' < j(\gamma)\), as desired.
\end{proof}
\end{thm}
\subsection{The universal property of Ketonen embeddings}
Since the Ultrapower Axiom implies the MO of the Ketonen order, the following theorem is an immediate consequence of \cref{KetonenMinimal}:
\begin{thm}[UA]
For any regular cardinal \(\delta\), there is a unique Ketonen ultrafilter at \(\delta\).
\end{thm}

But in fact Ketonen ultrafilters are minimum ultrafilters in a second way that we now describe. We first prove the key universal property that characterizes Ketonen embeddings under UA.
\begin{defn}
Suppose \(\delta\) is a regular uncountable cardinal. An ultrapower embedding \(j : V\to M\) is {\it zero order} at \(\delta\) if \(\sup j[\delta]\) is not tail uniform in \(N\).
\end{defn}
Thus a Ketonen embedding is a zero order ultrapower \(j\) with the additional property that \(M\) is generated by \(j[V]\cup \{\sup j[\delta]\}\).
\begin{thm}[UA; Universal Property of Ketonen embeddings]\label{KetonenUniversality}
Suppose \(\delta\) is regular, \(j : V\to M\) is Ketonen at \(\delta\), and \(j' : V\to M'\) is zero order at \(\delta\). Then there is an internal ultrapower embedding \(k : M\to M'\) such that \(j' = k\circ j\).
\begin{proof}
Let \((i,i') : (M,M')\to N\) be a comparison of \((j,j')\). Since \(j\) and \(j'\) are zero order, \(i\) and \(i'\) are continuous at \(\sup j[\delta]\) and \(\sup j'[\delta]\) respectively, so \[i(\sup j[\delta]) = \sup i \circ j[\delta] = \sup i'\circ j'[\delta] = i'(\sup j'[\delta])\] Since \(i\circ j = i'\circ j'\), \(i[j[V]] = i'[j'[V]]\subseteq i'[M']\). 

Since \(j\) is Ketonen, \(M = H^M(j[V]\cup \{\sup j[\delta]\})\), so since \(j[V]\cup \{\sup j[\delta]\}\subseteq i'[M']\), \(i[M]\subseteq i'[M']\). Define \(k : M\to M'\) by \[k = (i')^{-1}\circ i\]
Since \(i'\circ j' = i\circ j\), \(j' = (i')^{-1}\circ i\circ j = k \circ j\).

We must check that \(k\) is an internal ultrapower embedding of \(M\). To see that \(k\) is an ultrapower embedding of \(M\), we use the following trivial lemma (\cite{RF}, Lemma 5.14):
\begin{lma}
Suppose \(j : V\to M\) is an elementary embedding, \(j' : V\to M'\) is an ultrapower embedding, and \(k : M\to M'\) is an elementary embedding with \(k\circ j = j'\). Then \(k\) is an ultrapower embedding of \(M\).\qed
\end{lma}
To conclude that \(k\) is an internal ultrapower embedding of \(M\), we now cite another lemma, \cite{RF} Lemma 5.6:
\begin{lma}
Suppose \(i : M \to N\) is an internal ultrapower embedding, \(k : M\to M'\) is an ultrapower embedding, and \(i' : M'\to N\) is an elementary embedding such that \(i = i'\circ k\). Then \(k\) is an internal ultrapower embedding of \(M\).\qed
\end{lma}
We have produced an internal ultrapower embedding \(k : M \to M'\) such that \(j' = k\circ j\), which proves the theorem.
\end{proof}
\end{thm}

The supercompactness of the least strongly compact cardinal is actually a consequence of the universal property of Ketonen embeddings alone:

\begin{defn}
If \(\delta\) is a regular cardinal, \(\KU\) denotes the statement that the universal property of Ketonen embeddings holds at \(\delta\).
\end{defn}

\(\KU\) itself can be seen as a weak form of the Ultrapower Axiom:

\begin{thm}\label{KUA}
Suppose \(\delta\) is a regular cardinal. The following are equivalent:
\begin{enumerate}[(1)]
\item \(\KU\).
\item Suppose \(j_0 : V\to M_0\) is an ultrapower embedding and \(j_1 : V\to M_1\) is a Ketonen embedding at \(\delta\). Then the pair \((j_0,j_1)\) admits a comparison.
\end{enumerate}
\begin{proof}
(2) implies (1) by the proof of \cref{KetonenUniversality}. 

We show (1) implies (2). Assume (1) and fix \(j_0\) and \(j_1\) as in (2). Let \(\delta' = \text{cf}^{M_0}(\sup j_0[\delta])\). By \cref{KetonenExistence} applied in \(M_0\), in \(M_0\), there is a Ketonen embedding \(i_0 : M_0\to N\) at \(\delta'\). In particular, \(i_0\) is an internal ultrapower embedding of \(M_0\). 

We claim \(i_0\circ j_0 : V\to N\) is a zero order ultrapower embedding at \(\delta\). Note that \(\text{cf}^N(\sup i_0\circ j_0[\delta]) = \text{cf}^N(\sup i_0[\delta'])\): fix an increasing cofinal function \(f : \delta'\to \sup j_0[\delta]\) that belongs to \(M_0\), and note that \(i_0(f)\restriction \sup j_0[\delta]\) is an increasing cofinal function from \(\sup i_0[\delta']\) to \(\sup i_0\circ j_0[\delta]\) that belongs to \(N\). Since \(i_0\) is Ketonen at \(\delta'\), \(\sup i_0[\delta']\) is not tail uniform in \(N\). Hence by \cref{CofUniform}, \(\sup i_0\circ j_0[\delta]\) is not tail uniform in \(N\). Thus \(i_0\circ j_0\) is zero order at \(\delta\).

Therefore by \cref{KetonenUniversality}, there is an internal ultrapower embedding \(i_1 : M_1\to N\) such that \(i_0\circ j_0 = i_1\circ j_1\). In other words, \((i_0,i_1) : (M_0,M_1)\to N\) is a comparison of \((j_0,j_1)\).
\end{proof}
\end{thm}

The proof of \cref{KUA} yields the following partial analysis of the comparison of a Ketonen embedding with an arbitrary ultrapower embedding:

\begin{lma}[UA]\label{ZeroComparison}
Suppose \(\delta\) is a uniform regular cardinal and \(i: V\to M\) is an ultrapower embedding. Let \(\delta' = \textnormal{cf}^{M}(\sup i[\delta])\). Then there is a comparison of \((j_\delta,i)\) of the form \((i_*,j_{\delta'}^M) : (M_\delta,M)\to N\) where \(i_* : M_\delta\to N\) is an internal ultrapower embedding of \(M_\delta\).\qed
\end{lma}

Our theorems regarding the least supercompact cardinal require only \(\KU\), but we will just state them under full UA. What is interesting is that it does not seem possible to analyze larger supercompact cardinals using a principle analogous to \(\KU\). 

The following is a standard category theoretic argument:

\begin{thm}[UA]\label{KetonenUnique}
Suppose \(\delta\) is regular. There is a unique Ketonen embedding at \(\delta\).
\begin{proof}
Suppose \(j_0: V\to M_0\) and \(j_1: V\to M_1\) are Ketonen embeddings at \(\delta\). By \(\text{KU}_\delta\), there is an internal ultrapower embedding \(k_0 : M_0\to M_1\) such that \(j_1 = k_0\circ j_0\). By \(\text{KU}_\delta\), there is an internal ultrapower embedding \(k_1 : M_1\to M_0\) such that \(j_0 = k_1\circ j_1\). Since \(k_1 \circ k_0\) is an internal ultrapower embedding from \(M_0\) to itself, \(k_1 \circ k_0\) is the identity. Similarly \(k_1\circ k_0\) is the identity. Therefore \(M_0\) and \(M_1\) are isomorphic, and so since they are transitive classes, \(M_0 = M_1\) and \(k_0 = k_1 = \text{id}\). It follows that \(j_0 = k_0\circ j_0 = k_1\circ j_1 = j_1\), as desired.
\end{proof}
\end{thm}

\begin{defn}[UA]
Suppose \(\delta\) is regular. We denote by \(j_\delta : V \to M_\delta\) the unique Ketonen embedding at \(\delta\) and by \(U_\delta\) the unique Ketonen ultrafilter at \(\delta\).
\end{defn}

A very useful fact is that the internal ultrapower embeddings of \(M_\delta\) are characterized by the weakest possible property:

\begin{thm}[UA]\label{KetonenInternal}
Suppose \(\delta\) is regular. Suppose \(i : M_\delta\to N\) is an ultrapower embedding. Then the following are equivalent:
\begin{enumerate}[(1)]
\item \(i\) is continuous at \(\sup j_\delta[\delta]\).
\item \(i\) is an internal ultrapower embedding of \(M_\delta\).
\end{enumerate}
\begin{proof}
That (2) implies (1) does not require \(\KU\) (except to make sense of the notation \(M_\delta\)): since \(\sup j_\delta[\delta]\) is not tail uniform in \(M_\delta\), by \cref{TailUnifChar}, no internal ultrapower embedding of \(M_\delta\) is discontinuous at \(\sup j_\delta[\delta]\).

We now show that (1) implies (2). Assume (1). Let \(\delta_* = \sup j_\delta[\delta]\). Then \(i\circ j_\delta : V\to N\) is zero order at \(\delta\): since \(i\) is continuous at \(\delta_*\), \(\sup (i\circ j_\delta)[\delta] = i(\delta_*)\), and since \(\delta_*\) is not uniform in \(M_\delta\), \(i(\delta_*)\) is not uniform in \(N\). By \(\KU\), there is an internal ultrapower embedding \(k : M_\delta\to N\) such that \(k\circ j_{\delta} = i\circ j_\delta\). Since \(k\) is internal to \(M_\delta\), \(k\) is also continuous at \(\delta_*\), and therefore \(k(\delta_*]) = i(\delta_*)\). Thus \[k \restriction j_\delta[V]\cup \{\delta_*\} = i\restriction j_\delta[V]\cup \{\delta_*\}\] Since \(j_\delta : V\to M_\delta\) is Ketonen, \(M_\delta = H^{M_{\delta}}(j_\delta[V]\cup \{\delta_*\})\). It therefore follows that \(k = i\). Therefore \(i\) is an internal ultrapower embedding of \(M_\delta\) since \(k\) is. Thus (2) holds.
\end{proof}
\end{thm}

Another way to write this is the following:
\begin{cor}[UA]\label{CombinatorialInternal}
Suppose \(\delta\) is regular. Suppose \(W\) is an \(M_\delta\)-ultrafilter on \(X\in M_\delta\). Let \(\delta' = \textnormal{cf}^{M_\delta}(\sup j_\delta[\delta])\). Then \(W\in M_\delta\) if and only if for every partition \(P\) of \(X\) such that \(P\in M_\delta\) and \(|P|^{M_\delta} = \delta'\), there is some \(Q\subseteq P\) such that  \(Q\in M_\delta\), \(|Q|^{M_\delta} < \delta'\), and \(\bigcup Q\in W\).\qed 
\end{cor}

We omit the straightforward proof.

The following special case will be particularly important for us:
\begin{cor}[UA]\label{KetonenAbsorption}
Suppose \(\delta\) is regular. For any \(W\in \Un_{<\delta}\), \(j_W\restriction M_\delta\) is an internal ultrapower embedding of \(M_\delta\) and \(W\cap M_\delta\in M_\delta\).
\begin{proof}
This is immediate from \cref{KetonenInternal} as soon as one knows that \(j_W\restriction M_\delta\) is an ultrapower embedding of \(M_\delta\). This is a basic lemma proved in \cite{IR} Lemma 2.10.
\end{proof}
\end{cor}
\section{Closure properties of \(M_\delta\)}
\cref{KetonenAbsorption} shows that the Ketonen ultrapower \(M_\delta\) absorbs many countably complete ultrafilters from \(V\). In order to show that the embedding \(j_\delta :V\to M_\delta\) is strong, and even supercompact, we use large cardinals to convert this absorption of ultrafilters into the absorption of sets. This conversion process involves 
\subsection{Independent families}
\begin{defn}
Suppose \(X\) is a set and \(\kappa\) is an infinite cardinal. A set \(F\subseteq P(X)\) is called a {\it \(\kappa\)-independent family} of subsets of \(X\) if for any disjoint subfamilies \(\sigma\) and \(\tau\) of \(F\) of cardinality less than \(\kappa\), there is some \(x\in X\) that belongs to every element of \(\sigma\) and no element of \(\tau\).
\end{defn}

Independent families come up in the theory of filters in the following way:

\begin{lma}\label{IndependentFilter}
Suppose \(X\) is a set and \(\kappa\) is an infinite cardinal and \(F\subseteq P(X)\). Then \(F\) is a \(\kappa\)-independent family of subsets of \(X\) if and only if for any \(S\subseteq F\), \(S\cup \{X\setminus A : A\in F\setminus S\}\) generates a \(\kappa\)-complete filter on \(X\).\qed
\end{lma}

One could take \cref{IndependentFilter} to be the definition of a \(\kappa\)-independent family, but the advantage of our definition is that it makes the following absoluteness lemma  obvious:

\begin{lma}\label{IndAbsolute}
Suppose \(\kappa\) is a cardinal and \(M\) is an inner model such that \(M^{<\kappa}\subseteq M\). Suppose \(X\in M\) and \(M\) satisfies that \(F\) is a  \(\kappa\)-independent family of subsets of \(X\). Then \(F\) is a \(\kappa\)-independent family of subsets of \(X\).\qed
\end{lma}

An important combinatorial fact, due to Hausdorff, is the existence of \(\kappa\)-independent families:
\begin{thm}[Hausdorff]\label{Hausdorff}
Suppose \(\kappa\leq\lambda\) are cardinals such that \(\lambda^{<\kappa} = \lambda\). Then there is a \(\kappa\)-independent family \(F\) of subsets of \(\lambda\) such that \(|F| = 2^\lambda\).
\begin{proof}
It clearly suffices to construct a set \(X\) such that \(|X| = \lambda\) there is a \(\kappa\)-independent family \(F\) of subsets of \(X\) such that \(|F| = 2^\lambda\). Let \[X= \{(\sigma,S) : \sigma\in P_\kappa(\lambda)\text{ and }S\subseteq \sigma\}\]
For each \(A\subseteq \lambda\), let \(X_A = \{(\sigma,S) \in X : \sigma\cap A\in S\}\).
\end{proof}
\end{thm}

\subsection{The tight covering property}
In this section we prove that \(M_\delta\) has the {\it tight covering property} under a technical assumption proved in \cite{Frechet}.
\begin{defn}
An inner model \(M\) has the {\it tight covering property} at a cardinal \(\lambda\) if every set \(A\subseteq M\) such that \(|A|\leq \lambda\) is contained in a set \(A'\in M\) such that \(|A'|^M\leq \lambda\).
\end{defn}
The tight covering property at \(\lambda\) is simply the \((\lambda,\lambda)\)-covering property defined above. The key observation is the following fact:
\begin{thm}[UA]\label{CoverDichotomy}
Suppose \(\delta\) is a uniform regular cardinal. Then exactly one of the following holds:
\begin{enumerate}[(1)]
\item \(M_\delta\) has the tight covering property at \(\delta\).
\item For any \(U\in \Un_{\leq\delta}\), \(U\cap M_\delta\in M_\delta\).
\end{enumerate}
\begin{proof}
We first show that (1) implies (2) fails. Assume (1). Then \(\text{cf}^{M_\delta}(\sup j_\delta[\delta]) = \delta\). By \cref{CofUniform2}, since \(\sup j_\delta[\delta]\) is not tail uniform in \(M_\delta\), \(\delta\) is not uniform in \(M_\delta\). Since \(\delta\) is uniform in \(V\), it follows that (2) fails.

We now show that if (1) fails, then (2) holds. Let \(\delta' = \text{cf}^{M_\delta}(\sup j_\delta[\delta])\). By \cref{KetonenCov}, the tight covering property at \(\delta\) fails for \(M_\delta\) if and only if \(\delta' > \delta\). But as a consequence of \cref{CombinatorialInternal}, any countably complete \(M_\delta\)-ultrafilter on an ordinal \(\alpha < \delta'\) belongs to \(M_\delta\). In particular since \(\delta < \delta'\), (2) holds.
\end{proof}
\end{thm}

By \cref{CoverDichotomy}, there is a bizarre consequence of the failure of tight covering for \(M_\delta\): if \(M_\delta\) does not have the tight covering property at \(\delta\), then \(U_\delta\cap M_\delta \in M_\delta\), even though \(M_\delta = M_{U_\delta}\). But in fact, we do not know how to rule this out, and the following question is open:
\begin{qst}[UA]\label{TightCoverQ}
Suppose \(\delta\) is a regular cardinal. Does \(M_\delta\) have the tight covering property at \(\delta\)?
\end{qst}
Perhaps the situation is not as bizarre as it appears on first glance, considering the following theorem from \cite{IR}:

\begin{thm}
If \(\kappa\) is supercompact, then there is a \(\kappa\)-complete nonprincipal ultrafilter \(U\) on a cardinal such that \(U\cap M_U\in M_U\). 
\end{thm}

We do come very close to a positive answer to \cref{TightCoverQ}, and depending on one's interests, one might say that the question is answered in all cases of interest (e.g., assuming GCH).

We start with the following fact:

\begin{thm}[UA]\label{TightCover}
Suppose \(\delta\) is regular. Let \(\kappa = \textsc{crt}(j_\delta)\) and assume \(\kappa\) is \(\delta\)-strongly compact. Then \(M_\delta\) has the tight covering property at \(\delta\).
\begin{proof}Assume towards a contradiction that \(M_\delta\) does not have the tight covering propery, so by \cref{CoverDichotomy}, for any  \(U\in \Un_{\delta}\), \begin{equation}\label{BadSide}U\cap M_\delta\in M_\delta\end{equation} In order to control cardinal arithmetic, it is convenient to prove the following claim.
\begin{clm}\label{StrongCompactAbs}
In \(M_\delta\), \(\kappa\) is \(\delta\)-strongly compact.
\begin{proof}
By a theorem of Ketonen \cite{Ketonen} (which is closely related to \cref{KetonenRegularity}), to show that \(M_\delta\) satisfies that \(\kappa\) is \(\delta\)-strongly compact, it suffices to show that every ordinal \(\alpha\) such that \(\kappa \leq \alpha \leq\delta\) and \(\text{cf}^{M_{U_\delta}}(\alpha) \geq \kappa\) carries a uniform \(\kappa\)-complete ultrafilter in \(M_\delta\). But note that since \(M_\delta\) is closed under \(\kappa\)-sequences and \(\text{cf}^{M_{U_\delta}}(\alpha) \geq \kappa\), in fact \(\text{cf}(\alpha) \geq \kappa\). Therefore since \(\kappa\) is \(\delta\)-strongly compact, in \(V\), \(\alpha\) carries a \(\kappa\)-complete ultrafilter \(U\). Now \(U\cap M_\delta\in M_\delta\) by \cref{BadSide}, so \(U\cap M_\delta\) witnesses that \(\alpha\) carries a uniform \(\kappa\)-complete ultrafilter in \(M_\delta\).
\end{proof}
\end{clm}
Since \(\kappa\) is \(\delta\)-strongly compact in \(M_\delta\), by Solovay's theorem \cite{Solovay}, \((\delta^{<\kappa})^{M_\delta} = \delta\). Therefore we are in a position to apply \cref{Hausdorff} inside \(M_\delta\). This yields the existence of a \(\kappa\)-independent family \(F\) of subsets of \(\delta\) such that \(|F|^{M_\delta} = \delta\). (In fact one could also have taken \(|F|^{M_\delta} = (2^\delta)^{M_\delta}\), but we do not need this here.) By \cref{IndAbsolute}, \(F\) is a \(\kappa\)-independent family of subsets of \(\delta\) in \(V\).
\begin{clm}\label{PowerClm}
\(P(F)\subseteq M_\delta\).
\begin{proof}
Suppose \(S\subseteq F\). Let \(G\) be the filter generated by \(S\cup \{\delta\setminus A : A\in F\setminus S\}\). Then by \cref{IndependentFilter}, \(G\) is a \(\kappa\)-complete filter on \(\delta\). Moreover \(G\) is generated by \(\delta\)-many sets. Therefore since \(\kappa\) is \(\delta\)-strongly compact, \(G\) extends to a \(\kappa\)-complete ultrafilter \(U\) on \(\delta\). By \cref{BadSide}, \(U\cap M_\delta\in M_\delta\). But \(S = (U\cap M_\delta)\cap F\), so \(S\in M_\delta\).
\end{proof}
\end{clm}
Since \(P(F)\subseteq M_\delta\) and \(|F|^{M_\delta} = \delta\), it follows that \(P(\delta) \subseteq M_\delta\). But by \cref{BadSide}, \(U_\delta = U_\delta\cap M_\delta \in M_\delta\). This is a contradiction since \(M_\delta = M_{U_\delta}\) and no nonprincipal countably complete ultrafilter belongs to its own ultrapower.
\end{proof}
\end{thm}

A very similar argument yields the following fact:
\begin{thm}[UA]
Suppose \(\delta\) is regular. Let \(\kappa = \textsc{crt}(j_\delta)\) and assume \(\kappa\) is \(\delta\)-strongly compact. Then for any \(\gamma < \delta\), \(P(\gamma)\subseteq M_\delta\).
\end{thm}

Combining these facts, we obtain supercompactness:
\begin{thm}[UA]
Suppose \(\delta\) is regular. Let \(\kappa = \textsc{crt}(j_\delta)\) and assume \(\kappa\) is \(\delta\)-strongly compact. Then \((M_\delta)^\gamma\subseteq M_\delta\) for all \(\gamma < \delta\).
\begin{proof}
We may assume by induction that for all \(\bar \gamma < \gamma\), \((M_\delta)^{\bar \gamma}\subseteq M_\delta\). Then if \(\gamma\) is singular, it is immediate that \((M_\delta)^\gamma\subseteq M_\delta\). Therefore we may assume that \(\gamma\) is regular. 

Let \(\gamma' = \text{cf}^{M_\delta}(\sup j_\delta[\gamma])\). By the tight covering property, \(\gamma' \leq \delta\). But \(\gamma'\neq \delta\) since \(\text{cf}(\gamma') = \text{cf}(\sup j_\delta[\gamma]) = \gamma\). Therefore \(\gamma' < \delta\). By \cref{KetonenCov}, it follows that there is a set \(A'\in M_\delta\) containing \(j_\delta[\gamma]\) such that \(|A'|^{M_\delta}  = \gamma'\). Since \(\gamma' < \delta\), \(P(\gamma')\subseteq M_\delta\). Since \(|A'|^{M_\delta} = \gamma'\), \(P(A')\subseteq M_\delta\). Since \(j_\delta[\gamma]\subseteq A'\), it follows that \(j_\delta[\gamma]\in M_\delta\).
\end{proof}
\end{thm}

The following is a glaring open question:
\begin{qst}[UA]
Let \(\kappa = \textsc{crt}(j_\delta)\) and assume \(\kappa\) is \(\delta\)-strongly compact.  Is \(M_\delta\) is closed under \(\delta\)-sequences?
\end{qst}

We can answer this question when \(\delta\) is a successor cardinal:
\begin{thm}[UA]
Suppose \(\delta\) is regular but not strongly inaccessible. Let \(\kappa = \textsc{crt}(j_\delta)\) and assume \(\kappa\) is \(\delta\)-strongly compact. Then \(M_\delta\) is closed under \(\delta\)-sequences.
\begin{proof}
We use the following fact:
\begin{lma}
Suppose \(M\) is an ultrapower of \(V\) and \(\lambda\) is a cardinal. Then the following are equivalent:
\begin{enumerate}[(1)]
\item \(M^\lambda\subseteq M\).
\item \(M\) has the tight covering property at \(\lambda\) and \(P(\lambda)\subseteq M\).\qed
\end{enumerate}
\end{lma}
Therefore by \cref{TightCover}, it suffices to show that \(P(\delta)\subseteq M_\delta\). There are two cases:
\begin{case}\label{BigCofCase}
For some cardinal \(\gamma < \delta\) of cofinality at least \(\kappa\), \(2^\gamma\geq \delta\).
\end{case}
Note that by the proof of \cref{StrongCompactAbs} in the proof of \cref{TightCover}, \(\kappa\) is \({<}\delta\)-strongly compact in \(M_\delta\). Therefore by Solovay's theorem \cite{Solovay}, \((\gamma^{<\kappa})^{M_\delta} =\gamma\). By \cref{Hausdorff}, it follows that in \(M_\delta\), there is a \(\kappa\)-independent family \(F\) of subsets of \(\gamma\) such that \(|F|^{M_\delta} = \delta\). As in \cref{PowerClm} in the proof of \cref{TightCover}, \(P(F)\subseteq M_\delta\), so \(P(\delta)\subseteq M_\delta\).
\begin{case}
Otherwise.
\end{case}
In this case, \(\delta\) must be the successor of a singular cardinal of cofinality less than \(\kappa\). To see this, let \(\lambda < \delta\) be a cardinal such that \(2^\lambda\geq \delta\). Then \(2^{(\lambda^+)}\geq \delta\), and so since we are not in \cref{BigCofCase}, \(\lambda^+ = \delta\). Similarly since we are not in \cref{BigCofCase}, \(\text{cf}(\lambda) < \kappa\). We use the following standard fact:
\begin{lma}\label{SmallCofSuper} Suppose \(j : V\to M\) is an elementary embedding with critical point \(\kappa\). Suppose \(M\) is closed under \(\lambda\)-sequences. Then \(M\) is closed under \(\lambda^{<\kappa}\)-sequences.
\begin{proof}
Since \(|P_\kappa(\lambda)| = \lambda^{<\kappa}\), it suffices to show that \(j\restriction P_\kappa(\lambda)\in M\). For \(\sigma\in P_\kappa(\lambda)\), \(j(\sigma) = j[\sigma]\). But \(j\restriction \lambda \in M\), so the function \(\sigma\mapsto j[\sigma]\) belongs to \(M\).
\end{proof}
\end{lma}
By Konig's theorem \(\lambda^{<\kappa} \geq \lambda^+ = \delta\). Thus by \cref{SmallCofSuper}, \((M_\delta)^\delta\subseteq M_\delta\). 
\end{proof}
\end{thm}
\subsection{The least strongly compact cardinal}
We must now say something about the hypothesis that \(\textsc{crt}(j_\delta)\) is \(\delta\)-strongly compact. The following is Theorem 7.1 of \cite{Frechet}:
\begin{thm}[UA]
Suppose \(\delta\) is a successor cardinal or a strongly inaccessible cardinal. Then \(\textsc{crt}(j_\delta)\) is \(\delta\)-strongly compact.
\end{thm}

Thus we have the following facts:

\begin{thm}[UA]\label{LeastSuper}
Suppose \(\delta\) is a successor cardinal that carries a uniform countably complete ultrafilter. Then \(j_\delta : V\to M_\delta\) witnesses that \(\textsc{crt}(j_\delta)\) is \(\delta\)-supercompact.\qed
\end{thm}

\begin{thm}[UA]
Suppose \(\delta\) is an inaccessible cardinal that carries a uniform countably complete ultrafilter. Then \(j_\delta : V\to M_\delta\) witnesses that \(\textsc{crt}(j_\delta)\) is \({<}\delta\)-supercompact.\qed
\end{thm}

\begin{defn}
For any regular cardinal \(\delta\), let \(\kappa_\delta\) denote the least cardinal \(\kappa\) such that every regular cardinal in the interval \([\kappa,\delta]\) is uniform.
\end{defn}

We will have \(\kappa_\delta = \delta^+\) when \(\delta\) is not uniform.
\begin{lma}[UA]\label{crtchar}
Suppose \(\delta\) is a successor cardinal or a strongly inaccessible cardinal. If \(\delta\) is uniform then \(\textsc{crt}(j_\delta) = \kappa_\delta\).
\begin{proof}
It suffices to show that \(\textsc{crt}(j_\delta) \leq \kappa_\delta\). By \cref{KetonenRegularity}, every set \(A\subseteq M_\delta\) with \(|A|\leq \delta\) is contained in a set \(A'\in M_\delta\) such that \(|A'|^{M_\delta} < j_\delta(\kappa_\delta)\). Thus \(j_\delta(\kappa_\delta)> \delta\), so  \(\textsc{crt}(j_\delta) \leq \kappa_\delta\).
\end{proof}
\end{lma}
Using terminology inspired by Bagaria-Magidor, \(\kappa_\delta\) is the least cardinal that is \((\omega_1,\delta)\)-strongly compact. (Indeed it is the least \(\kappa\) such that every \(\kappa\)-complete filter on \(\delta\) generated by \(\delta\) sets extends to a countably complete ultrafilter.)
In particular, we have the following fact which is a version of Ketonen's theorem relating strongly compact cardinals to uniform ultrafilters on regular cardinals:
\begin{prp}\label{omega1Ketonen}
Assume \(\kappa\) is a cardinal. Then the following are equivalent:
\begin{enumerate}[(1)]
\item \(\kappa\) is the least \(\omega_1\)-strongly compact cardinal.
\item For all regular cardinals \(\delta \geq\kappa\), \(\kappa_\delta = \kappa\).\qed
\end{enumerate}
\end{prp}

A global consequence is the following fact:

\begin{thm}[UA]
The least \(\omega_1\)-strongly compact cardinal is supercompact.
\begin{proof}
Let \(\kappa\) be the least \(\omega_1\)-strongly compact cardinal. Fix a regular cardinal \(\delta \geq \kappa\). Then by \cref{omega1Ketonen} and \cref{crtchar}, \(\kappa = \kappa_\delta = \textsc{crt}(j_\delta)\) for all regular cardinals \(\delta\geq \kappa\). Therefore if \(\delta\geq \kappa\) is a successor cardinal, \(j_\delta: V\to M_\delta\) witnesses that \(\kappa\) is \(\delta\)-supercompact. Therefore \(\kappa\) is \(\delta\)-supercompact for all successor cardinals \(\delta\geq \kappa\). In other words, \(\kappa\) is supercompact.
\end{proof}
\end{thm}

\begin{cor}[UA]
The least strongly compact cardinal is supercompact.\qed
\end{cor}

\subsection{The next Fr\'echet cardinal in \(M_\delta\)}
In this subsection we establish a technical fact that will be of use in the proof of the Irreducibility Theorem below:
\begin{thm}[UA]\label{NextUniform}
Suppose \(\delta\) is a uniform regular cardinal. Let \(\kappa = \textsc{crt}(j_\delta)\), and assume \(\kappa\) is \(\delta\)-strongly compact. Suppose that in \(M_\delta\), \(\lambda\) is the least Fr\'echet cardinal above \(\delta\). Then \(\lambda\) is measurable.
\end{thm}

For the proof we need some facts about Fr\'echet cardinals from \cite{Frechet}.

\begin{defn}[UA]
For any Fr\'echet cardinal \(\lambda\), let \(U_\lambda\) denote the \(\sE\)-least Fr\'echet uniform ultrafilter on \(\lambda\). 
\end{defn}

Thus \(U_\lambda\) generalizes \(U_\delta\) to all Fr\'echet cardinals. It turns out that the following notion sometimes suffices as a substitute for regularity in the analysis of \(U_\lambda\):
\begin{defn}
A Fr\'echet cardinal \(\lambda\) is {\it isolated} if \(\lambda\) is a limit cardinal but \(\lambda\) is not a limit of Fr\'echet cardinals.
\end{defn}
A reasonable conjecture is that any isolated cardinal is measurable, but we do not know how to prove this from UA alone. (UA + GCH suffices.)
\begin{defn}
For any cardinal \(\gamma\), \(\gamma^\sigma\) denotes the least Fr\'echet cardinal \(\lambda\) such that \(\lambda > \gamma\).
\end{defn}

The following is a restatement of \cite{Frechet} Lemma 3.3:

\begin{prp}\label{IsolationLimit}
A cardinal \(\lambda\) is isolated if and only if \(\lambda = \gamma^\sigma\) for some cardinal \(\gamma\) with \(\gamma^+ < \lambda\).\qed
\end{prp}

Corollary 8.9 from \cite{Frechet} gives an example of how isolation can stand in for regularity in some of our arguments:
\begin{prp}[UA]\label{IsolatedInternal}
If \(\lambda\) is isolated then for any countably complete \(M_{U_\lambda}\)-ultrafilter \(W\) on an ordinal less than \(\sup j_{U_\lambda}[\lambda]\), \(W\in M_{U_\lambda}\).\qed
\end{prp}

\begin{proof}[Proof of \cref{NextUniform}]
First note that \(\lambda > \delta^+\): \(\delta\) is not uniform in \(M_\delta\) by \cref{TightCover}, and a theorem of Prikry \cite{Prikry} shows that if a regular cardinal \(\delta\) is not uniform, neither is its successor. Therefore by \cref{IsolationLimit}, \(\lambda\) is isolated. In other words, \(\lambda\) is a limit cardinal.

If \(M_\delta\) satisfies that there is a measurable cardinal in the interval \([\delta,\lambda]\), then \(\lambda\) is measurable. Let \(U = (U_\lambda)^{M_\delta}\). To show \(\lambda\) is measurable, it suffices to show that \(\textsc{crt}(j^{M_\delta}_{U}) \geq \delta\). Assume towards a contradiction that \(\textsc{crt}(j^{M_\delta}_{U}) < \delta\).

By  \cite{Frechet} Theorem 9.2 applied in \(M_\delta\), \(M_\delta\) satisfies that \(U\) is \(\gamma\)-supercompact for all Fr\'echet cardinals \(\gamma < \lambda\). Since \(\textsc{crt}(j_\delta)\) is \(\delta\)-strongly compact, \(\textsc{crt}(j_\delta)\) is \({<}\delta\)-strongly compact in \(M_\delta\) by \cref{KetonenInternal}, and hence \(U\) is \({<}\delta\)-supercompact in \(M_\delta\). By the Kunen inconsistency theorem, since \(\textsc{crt}(j^{M_\delta}_{U}) < \delta\) and \(U\) is \({<}\delta\)-supercompact in \(M_\delta\), we must have \(j^{M_\delta}_{U}(\delta) > \delta\).
\begin{clm}\label{InternalClm}
Suppose \(W\) is a countably complete \(M_{U}^{M_\delta}\)-ultrafilter on \(\delta\). Then \(W\in M_{U}^{M_\delta}\).
\begin{proof}
By \cref{IsolatedInternal} applied in \(M_\delta\), it suffices to show that \(W\in M_\delta\). Let \(i : M_{U}^{M_\delta}\to N\) be the ultrapower of \(M_{U}^{M_\delta}\) by \(W\) using functions in \(M_{U}^{M_\delta}\). Since \(\delta < j_{U}^{M_\delta}(\delta)\), and \(j_{U}^{M_\delta}(\delta)\) is regular in \(M_{U}^{M_\delta}\), \(i\) is continuous at \(j_{U}^{M_\delta}(\delta)\). Since \(\delta\) is not uniform in \(M_\delta\), \(j_{U}^{M_\delta}\) is continuous at \(\delta\). Therefore \(i\circ j_{U}^{M_\delta}\) is continuous at \(\delta\). It follows from \cref{KetonenInternal} that \(i\circ j_{U}^{M_\delta}\) is an internal ultrapower embedding of \(M_\delta\). Let \(\xi = [\text{id}]_{U}^{M_\delta}\). Then for any \(f\in M_\delta\),
\[i([f]_{U}^{M_\delta}) = i\circ j_{U}^{M_\delta}(f)(i(\xi))\]
Thus \(i\) can be defined within \(M_\delta\). Therefore \(W\in M_\delta\), as desired.
\end{proof}
\end{clm}
By the argument of \cref{TightCover}, since \(\textsc{crt}(j_\delta)\) is \(\delta\)-strongly compact, it follows from \cref{InternalClm}  that \(P(\delta)\subseteq M_{U_\lambda}^{M_\delta}\). Applying \cref{InternalClm} again, we have \(U_\delta = U_\delta\cap M_{U_\lambda}^{M_\delta}\in M_{U_\lambda}^{M_\delta}\subseteq M_\delta = M_{U_\delta}\), which contradicts the fact that a nonprincipal countably complete ultrafilter cannot belong to its own ultrapower.
\end{proof}

\section{The Irreducibility Theorem}
In this section, we prove the main result of this paper which we call the {\it Irreducibility Theorem}. To define the notion of irreduciblity, and to give the proof, we need some preliminary facts about factorizations of ultrapower embeddings.

\subsection{The Rudin-Frolik order}
The Rudin-Frolik order is an order on ultrafilters that measures how ultrapower embeddings factor into iterated ultrapowers.
\begin{defn}
The {\it Rudin-Frolik order} is defined on countably complete ultrafilters \(U\) and \(W\) by setting \(U\D W\) if there is an internal ultrapower embedding \(k : M_U\to M_W\) such that \(j_W = k\circ j_U\).
\end{defn}

The Ultrapower Axiom is equivalent to the statement that the Rudin-Frolik order is directed on countably complete ultrafilters.

\begin{defn}
A countably complete ultrafilter \(U\) is {\it irreducible} if for all \(D\D U\), either \(D\) is principal or \(D\) is isomorphic to \(U\).  

An ultrapower embedding \(j :V \to M\) is {\it irreducible} if \(j = j_U\) for some irreducible ultrafilter \(U\).
\end{defn}

One version of the Irreducibility Theorem we will prove is the following:
\begin{thm}[UA]
Suppose \(U\) is an irreducible Fr\'echet uniform ultrafilter on a successor cardinal or a strong limit singular cardinal \(\lambda\). Then \(M_U^\lambda\subseteq M_U\).
\end{thm}

We continue with some more preliminary facts about the Rudin-Frolik order.

\begin{defn}
Supposer \(U\in \Un\) and \(D\) is a nonprincipal countably complete ultrafilter such that \(D\D U\). Then \(U/D\) is the uniform \(M_D\)-ultrafilter derived from \(k\) using \([\text{id}]_U\) where \(k : M_D\to M_U\) is the (unique) internal ultrapower embedding of \(M_D\) such that \(j_U = k\circ j_D\).
\end{defn}

Thus \(j^{M_D}_{U/D} : M_D\to M_U\) witnesses \(D\D U\). The following is a key lemma in the analysis of the Rudin-Frolik order under UA, proved in \cite{RF} Lemma 6.6:

\begin{lma}\label{Pushdown}
Supposer \(U\in \Un\) and \(D\) is a nonprincipal countably complete ultrafilter such that \(D\D U\). Then in \(M_D\), \(j_D(U)\swo U/D\).\qed
\end{lma}

The following fact is one of the main theorems of \cite{RF}, Theorem 8.3:

\begin{thm}[UA]
A countably complete ultrafilter has at most finitely many predecessors in the Rudin-Frolik order up to isomorphism. 
\end{thm}

In other words, an ultrapower embedding of \(V\) has at most finitely many predecessors in the Rudin-Frolik order.

The following is an easy corollary:
\begin{thm}[UA]
For any ultrapower embedding \(j : V\to M\), there is a finite iterated ultrapower \[V = M_0\stackrel{j_0}{\longrightarrow} M_1 \stackrel{j_1}{\longrightarrow}\cdots \stackrel{j_{n-1}}{\longrightarrow} M_{n} = M\] such that for each \(i < n\), \(j_i : M_i\to M_{i+1}\) is an irreducible ultrapower embedding of \(M_i\) and \(j = j_{n-1}\circ \cdots j_1\circ j_0\).\end{thm}

The analysis of countably complete ultrafilters under UA therefore reduces to the analysis of irreducible ultrafilters (and how they can be iterated). But even given the Irreducibility Theorem, irreducible ultrafilters remain somewhat mysterious. 

\subsection{Indecomposability and factorization}
The key to the proof of the irreducibility theorem is a generalization of a theorem of Silver that under favorable circumstances allows an ultrapower embedding to be ``factored across a continuity point." 
\begin{defn}
Suppose \(\lambda\) is a cardinal. An ultrafilter \(U\) on a set \(X\) is {\it \(\lambda\)-indecomposable} if whenever \(P\) is a partition of \(X\) into \(\lambda\) parts, there is some \(Q\subseteq P\) such that \(|Q| < \lambda\) and \(\bigcup Q\in U\).
\end{defn}
An ultrafilter is \(\lambda\)-decomposable if it is not \(\lambda\)-indecomposable.
\begin{prp}
Suppose \(U\) is an ultrafilter and \(\lambda\) is a cardinal. Then the following are equivalent:
\begin{enumerate}[(1)]
\item \(U\) is \(\lambda\)-decomposable.
\item There is a Fr\'echet uniform ultrafilter \(D\) on \(\lambda\) with \(D\RK U\).
\end{enumerate}
Assuming \(U\) is countably complete, one can add to the list:
\begin{enumerate}[(3)]
\item There are elementary embeddings \(V\stackrel j\longrightarrow M\stackrel k\longrightarrow M_U\) such that \(j_U = k\circ j\) and \(\sup j[\lambda] \leq \textsc{crt}(k) < j(\lambda)\).\qed
\end{enumerate}
\end{prp}

The proof of the following theorem appears in \cite{Frechet} Theorem 4.8. A lucid sketch of the special case that was relevant to Silver appears in \cite{Silver}.

\begin{thm}[Silver]\label{Silver}
Suppose \(U\) is an ultrafilter and \(\delta\) is a regular cardinal. Assume \(U\) is \(\lambda\)-decomposable for all cardinals \(\lambda\) with \(\delta \leq \lambda \leq 2^\delta\). Then there is an ultrafilter \(D\) on some cardinal \(\gamma <\delta\) and an elementary embedding \(k : M_{D} \to M_U\) such that \(j_U((2^\delta)^+)\subseteq k[M_D]\).\qed
\end{thm}
In the countably complete case, which is the only case in which we will be interested, \(j_U((2^\delta)^+)\subseteq k[M_D]\) is equivalent to \(\textsc{crt}(k) > j_U((2^\delta)^+)\). Thus we have \( j_U((2^\delta)^+) = j_D((2^\delta)^+) = (2^\delta)^+\).

Under UA, this has the following consequence:
\begin{thm}[UA]\label{SigmaSilver}
Suppose \(U\) is a countably complete ultrafilter and \(\delta\) is a regular cardinal. Assume \(U\) is \(\lambda\)-decomposable for all cardinals \(\lambda\) with \(\delta \leq \lambda \leq 2^\delta\). Then there is a countably complete ultrafilter \(D\) on some cardinal \(\gamma <\delta\) and an internal ultrapower embedding \(k : M_{D} \to M_U\) such that \(\textsc{crt}(k) > (2^\delta)^+\).\end{thm}

For the proof we need the following theorem, which appears as \cite{Frechet} Theorem 12.1:

\begin{thm}[UA]\label{UFCounting}
For any cardinal \(\lambda\), \(|\Un_\lambda|\leq (2^\lambda)^+\).\qed
\end{thm}

\begin{proof}[Proof of \cref{SigmaSilver}]
By \cref{Silver}, fix an ultrafilter \(D\) on some cardinal \(\gamma <\delta\) and an elementary embedding \(k : M_{D} \to M_U\) such that \(j_U((2^\delta)^+)\subseteq k[M_D]\). Obviously \(D\) is countably complete since \(M_D\) embeds in \(M_U\).  

By \cref{kInternal}, to show \(k\) is an internal ultrapower embedding, it is enough to show that \(\tr U D\in k[M_D]\).  By \cref{UFCounting}, \(|\Un^{M_U}_{\leq j_U(\delta)}|^{M_U}\leq j_U((2^\delta)^+)\). Since \(j_U((2^\delta)^+)\subseteq k[M_D]\) and \(\Un^{M_U}_{\leq j_U(\delta)}\in k[M_D]\), it follows that \(\Un^{M_U}_{\leq j_U(\delta)}\subseteq k[M_D]\). But \(\tr U D\in \Un^{M_U}_{\leq j_U(\delta)}\) by \cref{BoundingLemma}. So \(\tr U D\in k[M_D]\) as desired.
\end{proof}

\begin{cor}[UA]\label{DeltaFactor}
Suppose \(\delta\) is a Fr\'echet uniform cardinal that is either a successor cardinal or a strongly inaccessible cardinal. Suppose \(i : M_\delta\to N\) is an internal ultrapower embedding. Then there is some \(D\in \Un_{<\delta}\) and an internal ultrapower embedding \(i': M_D^{M_\delta}\to N\) such that \(i = i'\circ j^{M_\delta}_D\).
\begin{proof}
Let \(U\in M_\delta\) be a countably complete ultrafilter such that \(j_U^{M_\delta} = i\). Then \(U\) is \(\lambda\)-indecomposable for every cardinal \(\lambda\) with \(\delta\leq \lambda \leq 2^\delta\), and indeed for every cardinal \(\lambda\) between \(\delta\) and the next measurable cardinal, simply because none of these cardinals are Fr\'echet uniform by \cref{NextUniform}. Applying \cref{SigmaSilver} in \(M_\delta\) yields the corollary.
\end{proof}
\end{cor}

\subsection{Combinatorics of normal fine ultrafilters}
A key ingredient in the proof of the Irreducibility Theorem is a pair of combinatorial lemmas regarding normal fine ultrafilters on \(P(\delta)\). The first lemma is due to Solovay, and the second is due to the author though it seems likely that it has already been discovered.
\begin{defn}
Suppose \(\delta\) is a regular cardinal and \(\vec S = \langle S_\alpha : \alpha < \delta\rangle\) is a stationary partition of \(S^\delta_\omega\). The {\it Solovay set} associated to \(\vec S\) is the set of all \(\sigma\subseteq \delta\) such that letting \(\gamma = \sup \sigma\), \(\sigma  = \{\alpha < \delta: S_\alpha\cap \gamma\text{ is stationary in } \gamma\}\). 
\end{defn}

\begin{thm}[Solovay]\label{Solovay}
Suppose \(\delta\) is a regular cardinal and \(\vec S\) is a stationary partition of \(S^\delta_\omega\) into \(\delta\) pieces. Then the Solovay set associated to \(\vec S\) belongs to every normal fine ultrafilter on \(P(\delta)\).\qed
\end{thm}
We omit the proof, which appears in \cite{MO}. In particular, there is a single set of cardinality \(\delta\) on which all normal fine ultrafilters on \(P(\delta)\) concentrate. Indeed, the function \(\sup : P(\delta)\to \delta\) is one-to-one on any Solovay set. Thus any normal fine ultrafilter \(\mathcal U\) on \(P(\delta)\) is {\it isomorphic} to the ultrafilter \(\sup_*(\mathcal U)\) on \(\delta\). 

The second lemma we need is less well-known.

\begin{lma}\label{NormalGeneration} Suppose \(\lambda\) is a cardinal and \(\mathcal F\) is a normal fine filter on \(P(\lambda)\). Suppose \(D\) is an ultrafilter on \(\lambda\). Let \(S = \{\sigma\subseteq j_D(\lambda) : [\textnormal{id}]_D\in \sigma\}\). Then \(j_D[\mathcal F]\cup \{S\}\) generates \(j_D(\mathcal F)\) in the sense that any \(X\in j_D(\mathcal F)\) contains \(j_D(\bar X)\cap S\) for some \(\bar X\in \mathcal F\).
\begin{proof}
Suppose \(X\in j_D(\mathcal F)\). Fix \(\langle X_\alpha : \alpha < \lambda\rangle\) such that letting \(\langle Y_\alpha : \alpha < j_D(\lambda)\rangle\) denote \(j_D(X_\alpha : \alpha < \delta\rangle)\), \(X = Y_{[\text{id}]_D}\). We may assume without loss of generality that \(X_\alpha\in \mathcal F\) for all \(\alpha < \lambda\). Let \(\bar X = \triangle_{\alpha < \delta} X_\alpha\). 

We claim \(j_D(\bar X)\cap S\subseteq X\). Suppose \(\sigma\subseteq \lambda\) and \(\sigma\in S\cap j_D(\bar X)\). The fact that \(\sigma\in S\) means \([\text{id}]_D\in \sigma\). The fact that \(\sigma\in j_D(\bar X)\) means \(\sigma\in \triangle _{\alpha < j_D(\delta)}Y_\alpha\). By the definition of the diagonal intersection, \(\sigma\in Y_\alpha\) for all \(\alpha\in \sigma\). Therefore \(\sigma\in Y_{[\text{id}]_D} = X\), as desired.
\end{proof}
\end{lma}

\begin{defn}
Suppose \(M\) is an inner model, \(X\) and \(A\subseteq P(X)\cap M\) are sets in \(M\), and \(\mathcal U\) is an \(M\)-ultrafilter on \(A\). We say \(\mathcal U\) is a {\it fine \(M\)-ultrafilter} every element of \(X\) belongs to \(\mathcal U\)-almost all elements of \(A\).
\end{defn}

\begin{cor}
Suppose \(\lambda\) is a cardinal, \(D\) is a countably complete ultrafilter on \(\lambda\), and \(\mathcal U\) is a normal fine ultrafilter on a set \(A\subseteq P(\lambda)\). Then \(j_D(\mathcal U)\) is the unique fine \(M_D\)-ultrafilter \(\mathcal U'\) on \(j_D(A)\) such that \(j_D[\mathcal U]\subseteq \mathcal U'\).\qed
\end{cor}

\subsection{Proof of the Irreducibility Theorem}

We use a slight variant of the notion of irreducibility:
\begin{defn}
Suppose \(\lambda\) is a cardinal and \(U\) is a countably complete ultrafilter. Then \(U\) is {\it \({<}\lambda\)-irreducible} if for all \(D\in \Un_{<\lambda}\) such that \(D \D U\), \(D\) is principal. An ultrapower embedding \(j : V\to M\) is {\it \({<}\lambda\)-irreducible} if for all \(D\in \Un_{<\lambda}\) such that \(j_D \D j\), \(D\) is principal.
\end{defn}

\begin{thm}[UA; Irreducibility Theorem]\label{IrredThm}
Suppose \(\delta\) is a uniform successor cardinal. Suppose \(i: V\to M\) is a \({<}\delta\)-irreducible ultrapower embedding. Then \(M^\delta\subseteq M\).
\end{thm}

To prove \cref{IrredThm}, one propagates the closure of \(M_\delta\) to \(M\) by comparing the two models. More specifically, the strategy of the proof of \cref{IrredThm} is to take the canonical comparison \((i_*,j) : (M_\delta,M)\to N\)  of \((j_\delta,i)\) and show that \(\textsc{crt}(i_*) > \delta\). It follows that \(N^\delta = (N^\delta)^{M_\delta}\subseteq N\) by applying in \(M_\delta\) the standard fact that the ultrapower by a \(\delta\)-complete ultrafilter is closed under \(\delta\)-sequences. But \(N\) is also an internal ultrapower of \(M\), so \(N\subseteq M\). In particular \(\text{Ord}^\delta\subseteq N\subseteq M\). But this implies \(i[\delta]\in M\), so \((M)^\delta\subseteq M\) by the standard criterion for closure of ultrapowers under \(\delta\)-sequences. 

\begin{proof}[Proof of \cref{IrredThm}]
Let \(\delta' = \text{cf}^M(\sup i[\delta])\) and let \(i_* : M_\delta\to N\) be the internal ultrapower embedding given by \cref{ZeroComparison} such that \[(i_*,j^M_{\delta'}) : (M_\delta,M)\to N\] is a comparison of \((j_\delta,i)\). By \cref{DeltaFactor}, there is some \(D\in \Un_{<\delta}\) and an internal ultrapower embedding \(i' : (M_D)^{M_\delta}\to N\) such that \(i_* = i'\circ j^{M_\delta}_D\) and \(\textsc{crt}(i') > j_D^{M_\delta}(\delta)\). 

To avoid superscripts, we introduce the following notation. Let \(d = j_D\) and let \(P = M_D\). Then \(j_D^{M_\delta} = d\restriction M_\delta\) and \(M_D^{M_\delta} = d(M_\delta) = (M_{d(\delta)})^P\) since \((M_\delta)^\delta\subseteq M_\delta\). Thus we have the commutative diagram \cref{Irred1Fig}.

\begin{figure}[htbp]
   \centering
\includegraphics[scale=.7]{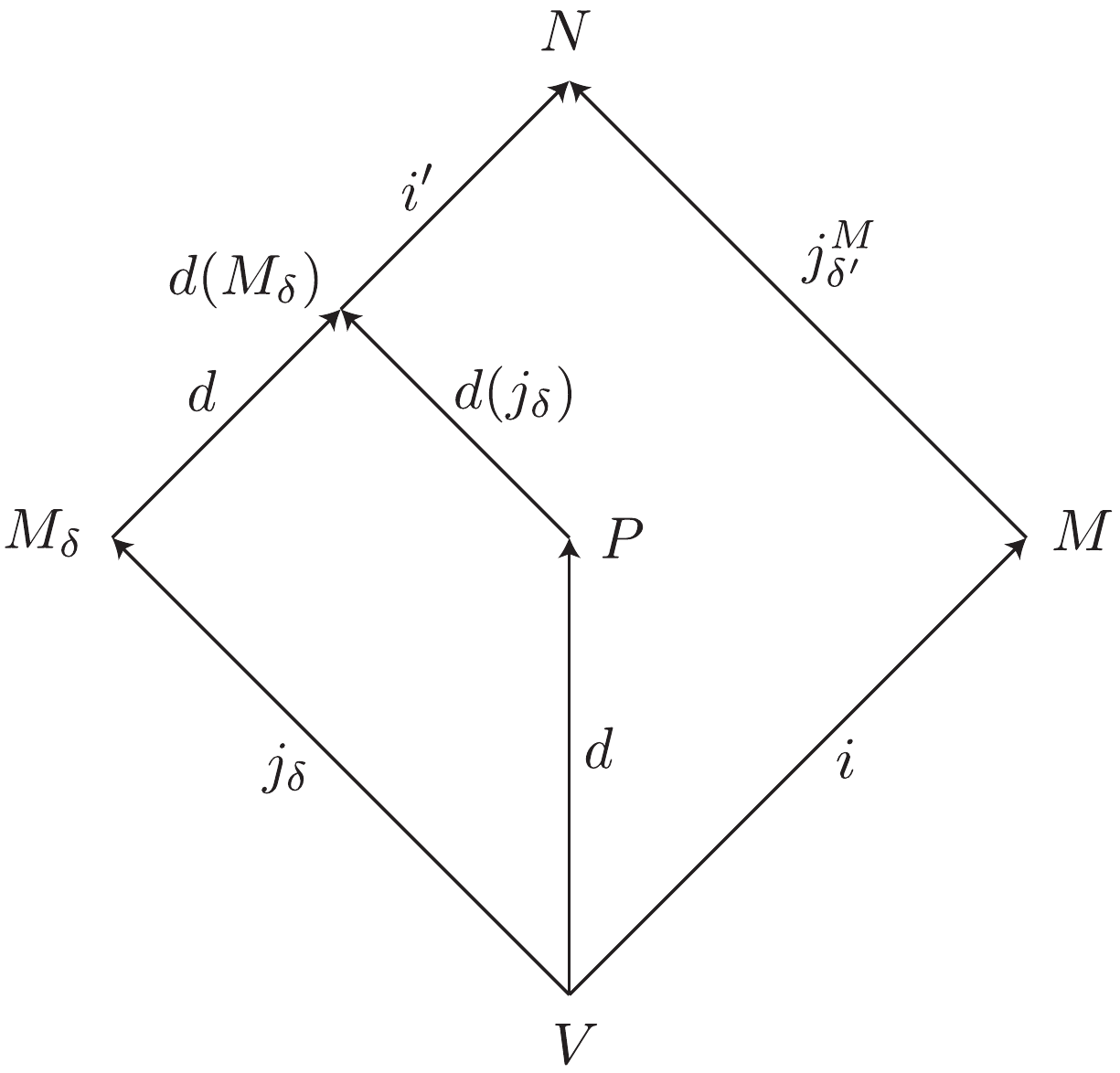}
   \caption{Comparing \((j_\delta,i)\), discovering \(d\).}
\label{Irred1Fig}
\end{figure}

The bulk of the proof is contained in the following claim: 
\begin{clm}\label{AgreementClm}
\(d(j_\delta) \restriction N = j_{\delta'}^M\restriction N\).
\end{clm}
\begin{proof}
We start by computing some closure properties of \(N\) relative to \(M\) and \(P\), leading to a proof that \(d(\delta) = \delta'\).

First, \(N\) is closed under \(d(\delta)\)-sequences relative to \(P\). This is because \(d(M_\delta)\) is closed under \(d(\delta)\)-sequences relative to \(P\) (by the elementarity of \(d\) and the fact that \(M_\delta\) is closed under \(\delta\)-sequences), and \(N\) is closed under \(d(\delta)\)-sequences relative to \(d(M_\delta)\) (since \(i'\) is an ultrapower embedding and \(\textsc{crt}(i') > d(\delta)\)). This facilitates the following calculation:

\begin{align} \text{cf}^N(\sup j_{\delta'}^M[\delta']) &= \text{cf}^N(\sup j^M_{\delta'}[\sup i[\delta]])\label{cofsup}\\
&= \text{cf}^N(\sup j^M_{\delta'}\circ i[\delta])\nonumber\\
&= \text{cf}^N(\sup i'\circ d(j_\delta)\circ d[\delta])\label{compid}\\
&= \text{cf}^N(\sup i'\circ d(j_\delta)[d(\delta)])\label{contd}\\
&= d(\delta)\label{sc}
\end{align}
\cref{cofsup} follows from the fact that \(\delta' = \text{cf}^M(\sup i[\delta])\). \cref{compid} follows from the fact that \(j^M_{\delta'}\circ i = i'\circ d(j_\delta)\circ d\). \cref{contd} follows from the fact that \(\sup d[\delta] = d(\delta)\), since \(\delta\) is regular and \(d = j_D\) for some \(D\in \Un_{<\delta}\). Finally \cref{sc} follows from the fact that \(i'\circ d(j_\delta)[d(\delta)]\in N\), since \(N\) is closed under \(d(\delta)\)-sequences relative to \(P\) and \(i'\circ d(j_\delta)\) is an internal ultrapower embedding of \(P\).

Second, \(N\) has the tight covering property at \(\delta'\) relative to \(M\). To see this, we use the contrapositive of \cref{TightCover}: if \(N = M_{\delta'}^M\) does not have the tight covering property at \(\delta'\), then \(\delta'\) is weakly inaccessible in \(M\) and \(\sup j_{\delta'}^M[\delta']\) is regular in \(N\).  But then \(\sup j_{\delta'}^M[\delta'] = \text{cf}^N(\sup j_{\delta'}^M[\delta']) = d(\delta)\). But \(\sup j_{\delta'}^M[\delta']\) is a limit cardinal of \(N\) since \(\delta'\) is weakly inaccessible in \(M\), while \(d(\delta)\) is a successor cardinal of \(P\) and hence a successor cardinal of \(N\) since \(\delta\) is a successor cardinal by assumption. This is a contradiction.

Now \(\text{cf}^N(\sup j_{\delta'}^M[\delta']) = \delta'\) by the tight covering property. Combining this with \cref{sc} above, \[d(\delta) =\delta'\]

Since \(d(\delta)\) is a successor cardinal in \(P\) and \(N\) is closed under \(d(\delta)\)-sequences relative to \(P\), \(d(\delta)\) is a successor cardinal in \(N\). Since \(N\subseteq M\), \(d(\delta)\) is a successor cardinal in \(M\). Therefore since \(N = (M_{d(\delta)})^M\), \cref{LeastSuper} applied in \(M\) implies that \(N\) is closed under \(\delta'\)-sequences relative to \(M\). We therefore have
\begin{equation}\label{MPAgreement}\text{Ord}^{d(\delta)}\cap P = \text{Ord}^{d(\delta)}\cap N = \text{Ord}^{d(\delta)}\cap M\end{equation}

We now work directly with ultrafilters instead of ultrapower embeddings. Recall that \(U_\delta\) denotes the ultrafilter derived from \(j_\delta\) using \(\sup j_\delta[\delta]\). Let \(U' = (U_{d(\delta)})^M\). We claim that \(U' = d(U_\delta)\). This at least makes sense because \(U'\subseteq P(d(\delta))\cap M\) and \(d(U_\delta)\subseteq P(d(\delta))\cap P\) and \(P(d(\delta))\cap M = P(d(\delta))\cap P\) by \cref{MPAgreement}.

Let \(W = D^-(U') = \{X\subseteq \delta : d(X)\in U'\}\). It is easy to see that \(W\) is a countably complete ultrafilter on \(\delta\).

Note that \(\{\alpha < \delta : \Un_\alpha = \emptyset\}\in W\) since \(\{\alpha < d(\delta) :\Un^{M}_\alpha = \emptyset\}\in U'\) and
\begin{align}d(\{\alpha <\delta : \Un_\alpha = \emptyset\})  &= \{\alpha < d(\delta) :\Un^{P}_\alpha = \emptyset\}\nonumber \\
&= \{\alpha < d(\delta):\Un^{N}_\alpha = \emptyset\}\label{UniformAbs0} \\
&= \{\alpha < d(\delta) :\Un^{M}_\alpha = \emptyset\} \label{UniformAbs1}
\end{align}
\cref{UniformAbs0} and \cref{UniformAbs1} follow from \cref{MPAgreement} and \cref{CombinatorialInternal}, which together imply \(\Un_{<d(\delta)}^P = \Un_{<d(\delta)}^N = \Un_{<d(\delta)}^P\).

Moreover, we claim \(W\) is weakly normal, in the sense that any regressive function on \(\delta\) is bounded below \(\delta\) on a \(W\)-large set. This follows from the fact that \(d\) is continuous at \(\delta\). Suppose \(f : \delta\to \delta\) is regressive. Then since \(U'\) is Ketonen in \(M\), by \cref{KetonenCombinatorial}, \(d(f)\) is bounded below \(d(\delta)\) on a \(U'\)-large set. But then there is some \(\bar \xi < \delta\) such that \(\xi < d(\bar \xi)\). Hence \(\{\alpha < d(\delta) : d(f)(\alpha) < d(\xi)\}\in U'\). Thus \(\{\alpha < \delta : f(\alpha) < \xi\}\in W\), so \(f\) is bounded below \(\delta\) on a \(W\)-large set.

By \cref{KetonenCombinatorial}, it follows that \(W\) is Ketonen at \(\delta\). Therefore by \cref{KetonenUnique}, \(W = U_\delta\). In other words,
\begin{equation}\label{dimage}
d[U_\delta]\subseteq U'
\end{equation}

Let \(\mathcal U\) be the normal fine ultrafilter on \(P(\delta)\) derived from \(j_\delta[\delta]\). Working in \(M\), let \(\mathcal U'\) be the normal fine ultrafilter on \(P(d(\delta))\cap M\) derived from \(j^M_{d(\delta)}\) using \(j^M_{d(\delta)}[d(\delta)]\). We want to use \cref{NormalGeneration} to show essentially that \(d(\mathcal U) = \mathcal U'\).

Fix a partition \(\vec S = \langle S_\alpha : \alpha < \delta\rangle\) of \(S^\delta_\omega\) into stationary sets. Let \(A\subseteq P(\delta)\) be the Solovay set associated to \(\vec S\). Thus \(A\) consists of those  \(\sigma \subseteq \delta\) such that \(\sigma = \{\xi < \gamma : S_\xi\text{ reflects to }\gamma\}\) where \(\gamma = \sup \sigma\). Now \(d(A)\) is the Solovay set defined from \(d(\vec S)\) in \(P\), but by \cref{MPAgreement}, \(d(A)\) is also the Solovay set defined from \(d(\vec S)\) in \(M\). By \cref{Solovay}, \(d(A)\in \mathcal U'\).

Since \(P(A)\cap P = P(A)\cap M\) and \(\mathcal U'\restriction d(A)\) is an normal fine \(M\)-ultrafilter on \(d(A)\), \(\mathcal U'\restriction d(A)\) is an normal fine \(P\)-ultrafilter on \(d(A)\).

Let \(s : \delta\to A\) be the inverse of the sup function on \(A\). Then \(s_*(U_\delta) = \mathcal U\restriction A\). Moreover \(d(s) : d(\delta)\to d(A)\) is the inverse of the sup function on \(d(A)\) so \(d(s)_*(U') = \mathcal U'\restriction d(A)\). Since \(d[U_\delta]\subseteq U'\), if \(X\in \mathcal U\restriction A\), then \(s^{-1}[X]\in U_\delta\), so \(d(s^{-1}[X])\in U'\), so \(X\in \mathcal U'\restriction d(A)\). Thus \[d[\mathcal U\restriction A]\subseteq \mathcal U'\restriction d(A)\]
Since \(\mathcal U'\) is fine in \(M\), letting \(S = \{\sigma\in d(A) : [\text{id}]_D\in \sigma\}\), \(S\in \mathcal U'\restriction d(A)\). But by \cref{NormalGeneration}, \(d(\mathcal U\restriction A)\) is generated by \(d[\mathcal U\restriction A]\cup \{S\}\). Thus \(d(\mathcal U\restriction A) =  \mathcal U'\restriction d(A)\). But then \[U' = \textstyle\sup_*(\mathcal U'\restriction d(A)) = \sup_*(d(\mathcal U\restriction A)) = d(\sup_*(\mathcal U\restriction A)) = d(U_\delta)\]
Thus \(U' = d(U_\delta)\), as claimed.

Finally, applying \cref{MPAgreement} again, \[d(j_\delta) \restriction N  = j_{d(U_\delta)}^N = j_{U'}^N = j_{\delta'}^M \restriction N \]
This proves the claim.
\end{proof}

The following claim will easily imply the theorem:
\begin{clm}\label{RFClm}
\(d\D i\).
\end{clm}
\begin{proof}[Proof of \cref{RFClm}]
Let \((k,\ell) : (P,M)\to Q\) be the canonical comparison of \((d,i)\).

Since \((k,\ell)\) is the canonical comparison of \((d,i)\), by \cref{Pushout}, \((k,\ell)\) is the pushout of \((d,i)\). Therefore since \((i' \circ d(j_\delta),j_{\delta'}^M) : (P,M)\to N\) is also a comparison of \((d,i)\), there is an internal ultrapower embedding \(h : Q\to N\) such that \(j_{\delta'}^M = h\circ \ell\). Thus we have the commutative diagram \cref{Irred2Fig}.

\begin{figure}[htbp]
   \centering
\includegraphics[scale=.7]{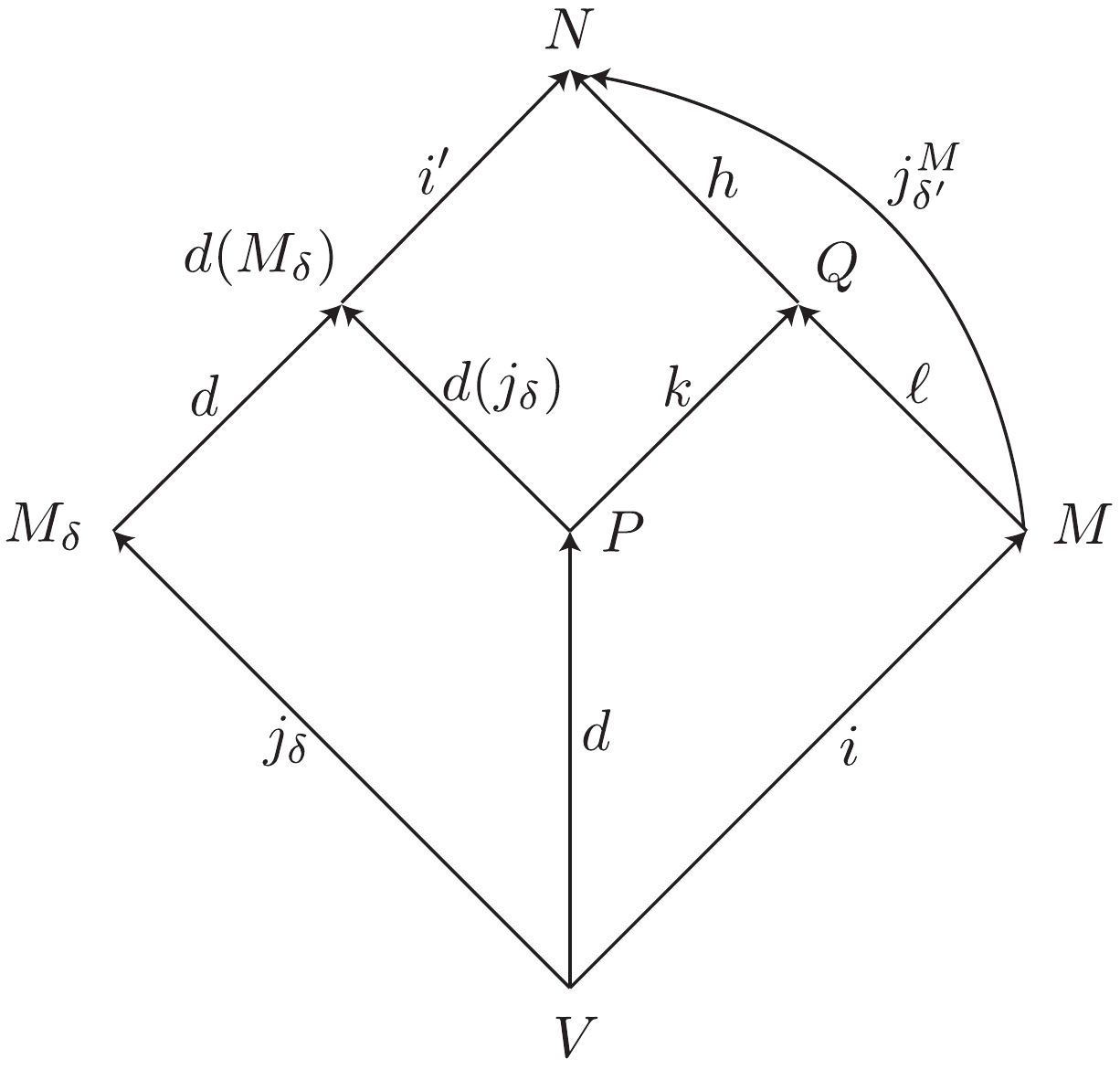}
   \caption{Comparing \((d,i)\), discovering \(k\).}
\label{Irred2Fig}
\end{figure}

Since \(j_{\delta'}^M\) is an irreducible ultrapower embedding of \(M\) and \(j_{\delta'}^M = h\circ \ell\), either \(\ell\) or \(h\) is the identity.

\begin{case} \(\ell\) is the identity.\end{case}
By case hypothesis, we have that \(Q = M\) and \(k :P\to M\) is an internal ultrapower embedding of \(P\) such that \(i =\ell\circ i = k\circ d\). Thus \(d\D i\). 

\begin{case} \(h\) is the identity.\end{case}
We will show in this case that in fact \(N = M\) and \(j_{\delta'}^M\) is the identity. By case hypothesis, we have the commutative diagram \cref{Irred3Fig}.

\begin{figure}[htbp]
   \centering
\includegraphics[scale=.7]{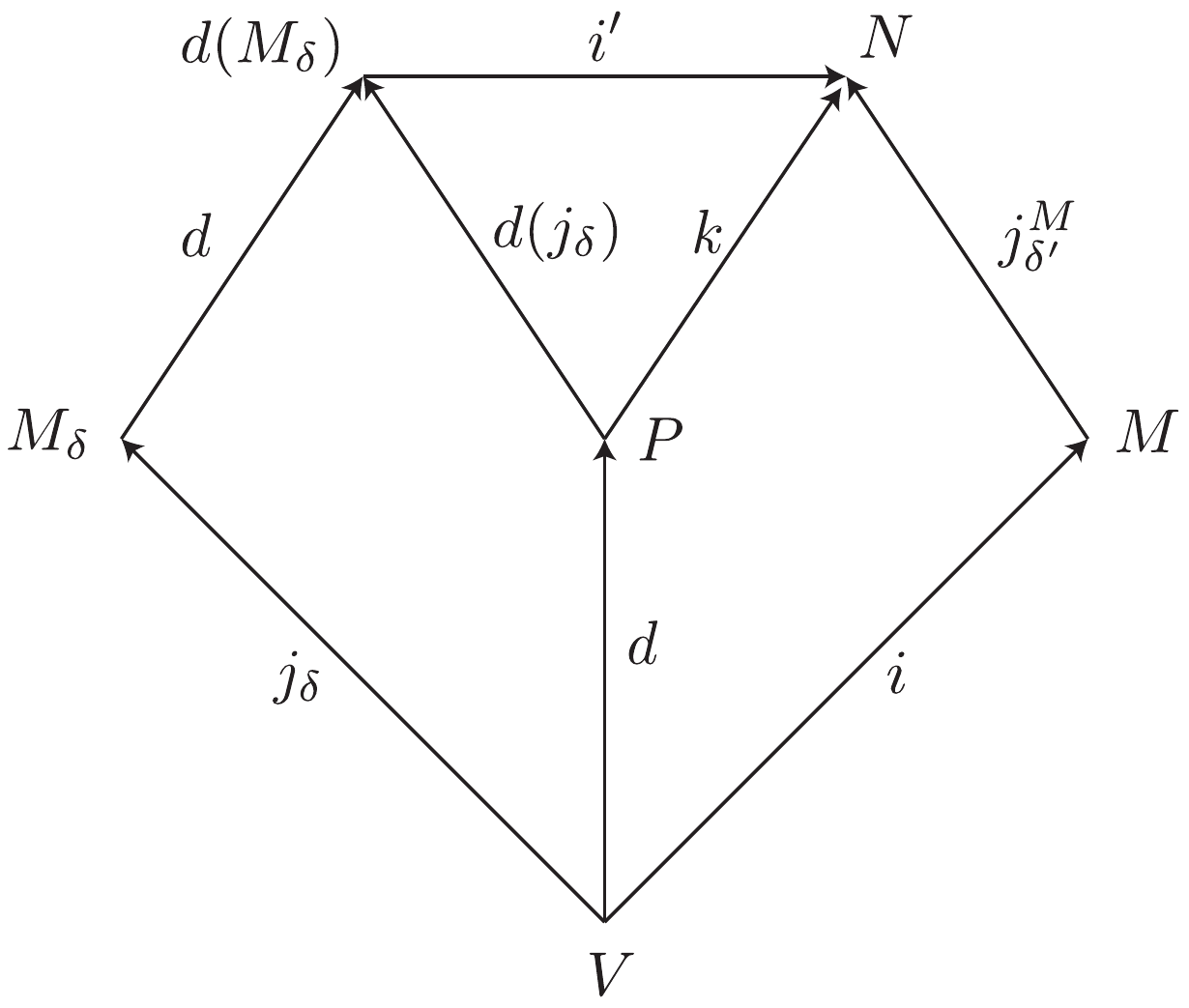}
   \caption{The case \(h = \text{id}\).}
\label{Irred3Fig}
\end{figure}

Let \(e = d(j_\delta) \restriction N = j_{\delta'}^M\restriction N\), using \cref{AgreementClm}. Then since \(e = d(j_\delta)\restriction N\), \(e\) is amenable to \(P\), and since \(e = j_{\delta'}^M\restriction N\), \(e\) is amenable to \(M\). Therefore by \cref{CanonicalInternal}, \(e\) is an internal ultrapower embedding of \(N\). But \(N = (M_{\delta'})^M\), so if \(j_{\delta'}^M\restriction N\) is amenable to \(N\), then relative to \(M\), \((M_{\delta'})^M\) is closed under \(\alpha\)-sequences for all ordinals \(\alpha\). In other words, \(N = M\) and so \(j_{\delta'}^M\) is the identity. But then once again \(k : P \to M\) is an internal ultrapower embedding of \(P\) and \(i =j_{\delta'}^M\circ i = k\circ d\).
\end{proof}

Now \(d \D i\), but \(d = j_D\) for some \(D\in \Un_{<\delta}\) and \(i\) is \({<}\delta\)-irreducible, so \(d\) is the identity. Therefore \(i_* = i'\circ d = i'\). Thus \(\textsc{crt}(i_*) = \textsc{crt}(i') > d(\delta) = \delta\). Since \(i_* : M_\delta\to N\) is an ultrapower embedding with critical point above \(\delta\), \(N\) is closed under \(\delta\)-sequences in \(M_\delta\). By \cref{LeastSuper}, \(M_\delta\) is itself closed under \(\delta\)-sequences in \(V\), so it follows that \(N\) is truly closed under \(\delta\)-sequences. But \(N\) is an internal ultrapower of \(M\), so \(N\subseteq M\). Hence \(\text{Ord}^\delta\subseteq M\), so \(i[\delta]\in M\). Since \(i : V\to M\) is an ultrapower embedding and \(i[\delta]\in M\), \(M^\delta\subseteq M\), as desired. 
\end{proof}

\section{Supercompactness}
In this section, we use the Irreducibility Theorem to prove the main theorem of this paper:
\begin{thm}[UA]\label{MenasGlobal}
If \(\kappa\) is a cardinal, the following are equivalent:
\begin{enumerate}[(1)]
\item \(\kappa\) is strongly compact.
\item \(\kappa\) is supercompact or a measurable limit of supercompact cardinals.
\end{enumerate}
\end{thm}
This is achieved by proving a more local result.
\subsection{Level-by-level equivalence at successor cardinals}
The title of this subsection comes from the following result:
\begin{thm}[UA]\label{MenasLocal}
If \(\kappa\) is a cardinal and \(\delta\) is a successor cardinal, then the following are equivalent:
\begin{enumerate}[(1)]
\item \(\kappa\) is \(\delta\)-strongly compact.
\item \(\kappa\) is \(\delta\)-supercompact or a measurable limit of \(\delta\)-supercompact cardinals.
\end{enumerate}
\end{thm}

\cref{MenasLocal} easily implies the main theorem of the paper:

\begin{proof}[Proof of \cref{MenasGlobal} from \cref{MenasLocal}]
The fact that (2) implies (1) is due to Menas \cite{Menas} and does not require UA.

We now prove the converse. Suppose \(\kappa\) is strongly compact. If \(\kappa\) is supercompact we are done, so suppose \(\kappa\) is not supercompact. For each successor cardinal \(\delta\), let \(A_\delta\subseteq \kappa\) be the set of \(\delta\)-supercompact cardinals less than \(\kappa\). By \cref{MenasLocal}, \(A_\delta\) is unbounded in \(\kappa\) for all successor cardinals \(\delta\). Moreover, if \(\delta_0 \leq \delta_1\), \(A_{\delta_0}\supseteq A_{\delta_1}\). Therefore there is some set \(A\subseteq \kappa\) such that for all sufficiently large \(\delta\), \(A_\delta = A\). It follows that \(A\) is the set of supercompact cardinals below \(\kappa\). But since \(A = A_\delta\) for some \(\delta\), \(A\) is unbounded in \(\kappa\). Therefore \(\kappa\) is a limit of supercompact cardinals, as desired.
\end{proof}

\cref{MenasLocal} is proved by analyzing the following ultrafilters using \cref{IrredThm}.

\begin{defn}[UA]
If \(\delta\) is a regular cardinal and \(\nu \leq \delta\) is a cardinal, then \(U^\nu_\delta\) denotes the \(\sE\)-least \(\nu\)-complete uniform ultrafilter on \(\delta\) if it exists.
\end{defn}

This analysis requires two simple lemmas:
\begin{lma}[UA]\label{SuccIrred}
Suppose \(\nu\) is a successor cardinal and \(\delta > \nu\) is a regular cardinal. Then \(U^{\nu}_\delta\) is irreducible if it exists.
\begin{proof}
Suppose \(D\D U\), and we will show that either \(D\) is isomorphic to \(U\) or \(D\) is principal. By replacing \(D\) with an isomorphic ultrafilter, we may assume \(D\in \Un\) and \(D\E U\). Given this, we will show that either \(D = U\) or \(D\) is principal.  

\begin{case} \(\textsc{sp}(D) = \delta\)\end{case} 
Then \(D\) is a \(\nu\)-complete ultrafilter \(U\) on \(\delta\), so \(U\E D\). Therefore \(U = D\).
\begin{case} \(\textsc{sp}(D) < \delta\).\end{case} 
Then since \(\delta\) is regular, \(j_D(\delta) = \sup j_D[\delta]\). It follows that \(\textsc{sp}(U/D) = j_D(\delta)\) since \(\textsc{sp}(U/D)\) is derived from the internal ultrapower embedding \(k : M_D\to M_U\) using \([\text{id}]_U\), and by the uniformity of \(U\), \(\sup j_D[\delta]\) is the least possible ordinal mapped above \([\text{id}]_U\) by \(k\). Moreover \[\textsc{crt}(U/D) \geq \textsc{crt}(U) \geq \nu\] so \(U/D\) is \(\nu\)-complete in \(M_D\). But also, and this is a key point, \[j_D(\nu) \leq j_U(\nu) = \nu\] since \(\nu\) is a successor cardinal. Therefore \(M_D\) satisfies that \(U/D\) is a uniform \(j_D(\nu)\)-complete ultrafilter on \(j_D(\delta)\). Therefore \(j_D(U)\E U/D\) in \(M_D\). By \cref{Pushdown}, it follows that \(D\) is principal.
\end{proof}
\end{lma}

Our second lemma shows that the requirement that \(\nu\) is a successor cardinal is necessary above.

\begin{lma}[UA]\label{NormalTrans}
Suppose \(\kappa\) is a measurable cardinal and \(\delta > \kappa\) is a uniform regular cardinal. Let \(U = U^\kappa_\delta\) and let \(D = U^\kappa_\kappa\). Then in \(M_D\), \(\tr D U = U^{\kappa^+}_{j_D(\delta)}\).
\begin{proof}
Note that \(D = U^\kappa_\kappa\) is the unique normal ultrafilter on \(\kappa\) of Mitchell order zero.

Suppose that in \(M_D\), \(Z\) is a \(\kappa^+\)-complete uniform ultrafilter on \(j_D(\delta)\). Then \(D^-(Z)\) is a \(\kappa\)-complete uniform ultrafilter on \(\delta\), so \(U \E D^-(Z)\). Hence by \cref{OrderPreserving}, \(M_D\) satisfies \[\tr D U \E \tr D {D^-(Z)} \E Z\]
Thus  in \(M_D\), \(\tr D U\) lies \(\E\)-below every \(\kappa^+\)-complete uniform ultrafilter on \(j_D(\delta)\), so to show \(\tr D U = U^{\kappa^+}_{j_D(\delta)}\), it suffices to show that \(\tr D U\) is  \(\kappa^+\)-complete.

Since \(D\) is a normal ultrafilter, either \(D\D U\) or \(D\mo U\). 
If \(D\mo U\) then \(\tr D U = j_D(U)\) so \(\tr D U\) is \(j_D(\kappa)\)-complete; since \(\kappa^+\leq j_D(\kappa)\), we are done. 

If \(D\D U\), then obviously \(\tr D U = U/D\) is \(\kappa\)-complete. Since \(\kappa\) is not measurable in \(M_D\), \(\tr D U\) is \(\kappa^+\)-complete, as desired.
\end{proof}
\end{lma}

\begin{proof}[Proof of \cref{MenasLocal}]
The fact that (2) implies (1) is due to Menas \cite{Menas} and does not require UA.

We now prove the converse. Let \(U = U^{\delta}_\kappa\) and let \(D = U^\kappa_\kappa\). Working in \(M_D\), let \(Z = \tr D U\). By \cref{NormalTrans}, \(Z = U^{\kappa^+}_{j_D(\delta)}\) so by \cref{SuccIrred}, \(Z\) is an irreducible \(\kappa^+\)-complete uniform ultrafilter on \(j_D(\delta)\). Let \(\kappa' = \textsc{crt}(Z)\). Then by \cref{IrredThm}, \(Z\) witnesses that \(\kappa'\) is \(j_D(\delta)\)-supercompact in \(M_D\). Moreover since \(Z \E j_D(U) = U^{j_D(\kappa)}_{j_D(\delta)}\) by \cref{BoundingLemma}, \(\kappa' \leq \textsc{crt}(j_D(U))\) with equality if and only if \(Z = j_D(U)\). Thus in \(M_D\) there is a \(j_D(\delta)\)-supercompact cardinal in the interval \([\kappa, j_D(\kappa)]\). By a standard reflection argument, either \(\kappa\) is supercompact or \(\kappa\) is a measurable limit of supercompact cardinals.
\end{proof}

\subsection{Level-by-level equivalence at singular cardinals}
In this section, we tackle the question of the local equivalence of strong compactness and supercompactness at a singular cardinal \(\lambda\). This depends on the cofinality of \(\lambda\) in the following way.
\begin{lma}
Suppose \(\kappa < \lambda\) and \(\textnormal{cf}(\lambda) < \kappa\). Then \(\kappa\) is \(\lambda\)-strongly compact if and only if \(\kappa\) is \(\lambda^+\)-strongly compact, and  \(\kappa\) is \(\lambda\)-supercompact if and only if \(\kappa\) is \(\lambda^+\)-supercompact.
\end{lma}
Thus \cref{MenasLocal} implies the following fact:
\begin{thm}[UA]
Suppose \(\kappa < \lambda\) and \(\textnormal{cf}(\lambda) < \kappa\). Then \(\kappa\) is \(\lambda\)-strongly compact if and only if \(\kappa\) is \(\lambda\)-supercompact or a measurable limit of \(\lambda\)-supercompact cardinals.
\end{thm}
When \(\lambda\) is singular of large cofinality, equivalence provably fails. For example, we have the following fact:

\begin{cor}
Let \((\kappa_0,\lambda_0)\) be the lexicographically least pair \((\kappa, \lambda)\) such that \(\kappa < \lambda\), \(\kappa\) is \(\lambda\)-strongly compact, and \(\lambda\) is a strong limit cardinal of cofinality at least \(\kappa\). Then \(\kappa_0\) is \(\lambda_0\)-strongly compact but not \(\lambda_0\)-supercompact.\qed
\end{cor}

For more on this, see \cite{Apter}. The question we consider here is whether there is a notion of strong compactness (i.e., a filter extension property) at singular cardinals for which level-by-level equivalence with supercompactness holds (in the sense that Menas's theorem can be reversed).  

\begin{thm}[UA]\label{LBL}
Suppose \(\kappa  < \lambda\) and \(\lambda\) is singular. Then the following are equivalent:
\begin{enumerate}[(1)]
\item There is a \(\kappa\)-complete ultrafilter on \(P_\kappa(\lambda)\) extending the club filter.
\item \(\kappa\) is \(\lambda\)-supercompact or a measurable limit of \(\lambda\)-supercompact cardinals.
\end{enumerate}
\end{thm}

To prove this we use a lemma that appears as \cite{GCH} Lemma 3.3:

\begin{lma}\label{ClubLemma}
Suppose \(U\) is a countably complete ultrafilter, \(\lambda\) is a cardinal, and \(j_U[\lambda]\subseteq A\subseteq j_U(\lambda)\) has the property that \(j_U(f)[A]\subseteq A\) for every \(f : \lambda\to \lambda\). Suppose \(D\) is a countably complete ultrafilter on an ordinal \(\gamma < \lambda\). Suppose \((k,i) : (M_D,M_U)\to N\), \(k\) is an internal ultrapower embedding, and \(k\circ j_D = i\circ j_U\). Then \(k([\textnormal{id}]_D) \in i(A)\).
\end{lma}

\cref{ClubLemma} comes into the picture through the following lemma.
\begin{lma}\label{Club2}
Suppose \(\mathcal U\) extends a normal filter \(\mathcal F\) on \(P(\lambda)\) and \(D\) is a countably complete ultrafilter on \(\gamma < \lambda\) such that \(D\sD \mathcal U\). Let \(\mathcal U' = \mathcal U/D\). Then \(\mathcal U'\) extends \(j_D(\mathcal F)\).
\end{lma}

\begin{proof}
It is easy to see that \(j_D[\mathcal U]\subseteq \mathcal U'\). In particular \(j_D[\mathcal F]\subseteq \mathcal U'\).

By \cref{ClubLemma}, \(k([\text{id}]_D)\in [\text{id}]_\mathcal U\): note that \(j_\mathcal U(f)[ [\text{id}]_\mathcal U]\subseteq  [\text{id}]_\mathcal U\) for all \(f : \lambda\to \lambda\), since this merely says that \( [\text{id}]_\mathcal U\) belongs to \(j_D(C)\) where \(C\subseteq P(\lambda)\) is the club of closure points of \(f\). Letting \(S = \{A\in j_D(P(\lambda)) : [\text{id}]_D\in A\}\), we have \([\text{id}]_\mathcal U\in k(S)\) and so \(S\in \mathcal U'\). By \cref{NormalGeneration}, \(j_D(\mathcal F)\subseteq \mathcal U\).
\end{proof}

\begin{prp}[UA]\label{IrredExt}
For any normal fine filter \(\mathcal F\) on \(P(\lambda)\) and any successor cardinal \(\nu < \lambda\), the \(\swo\)-least \(\nu\)-complete ultrafilter \(U\) that is isomorphic to an extension of \(\mathcal F\) is \({<}\lambda\)-irreducible.
\begin{proof}
Suppose \(D\) is an ultrafilter on \(\gamma < \lambda\) and \(D\sD U\). Let \(U' = U/D\). To show \(D\) is principal, it suffices by \cref{Pushdown} to show that \(j_D(U)\E U'\). Note that \(U'\) is isomorphic to \(\mathcal U' = \mathcal U/D\). But by \cref{Club2}, \(\mathcal U'\) extends \(j_D(\mathcal F)\). Moreover \(U'\) is \(j_D(\nu)\)-complete since \(j_D(\nu) = \nu\) and \(\textsc{crt}(U') \geq \textsc{crt}(U) \geq \nu\). Therefore \(U'\) is a \(j_D(\nu)\)-complete extension of \(j_D(\mathcal F)\), and so  \(j_D(U)\E U'\), as desired.
\end{proof}
\end{prp}

\begin{proof}[Proof of \cref{LBL}]
The proof that (2) implies (1) follows the proof of Menas's theorem \cite{Menas} and does not require UA.

We now show (1) implies (2). Assume (1).

Note that the filter \(\mathcal F\) on \(P(\lambda)\) generated by the club filter on \(P_\kappa(\lambda)\) is normal. 

Suppose first that for some cardinal \(\nu < \kappa\), the least \(\nu\)-complete ultrafilter \(U\) that is isomorphic to an extension of \(\mathcal F\) is \(\kappa\)-complete. By replacing \(\nu\) with \(\nu^+\), we may assume without loss of generality that \(\nu\) is a successor cardinal (notice that \(\nu^+ < \kappa\)). Then \(U\) is \({<}\lambda\)-irreducible by \cref{IrredExt}. It follows from \cref{IrredThm} that \(U\) witnesses that \(\kappa\) is \({<}\lambda\)-supercompact. Since \(\lambda\) is singular, \(U\) witnesses that \(\kappa\) is \(\lambda\)-supercompact, so (2) holds.

Suppose instead that for each \(\nu < \kappa\), the least \(\nu\)-complete ultrafilter \(U\) that is isomorphic to an extension of \(\mathcal F\) is \(\kappa\)-complete. Then a similar argument shows that \(\kappa\) is a limit of \(\lambda\)-supercompact cardinals, so (2) holds.
\end{proof}

\subsection{Ultrafilters on inaccessible cardinals}\label{Anom}
In this short subsection we discuss the issues with establishing a version of \cref{MenasLocal} when \(\delta\) is an inaccessible cardinal. The first thing we show is that the only obstruction is the fact that we do not know how to analyze \(j_\delta: V \to M_\delta\) when \(\delta\) is inaccessible.
\begin{prp}[UA]\label{InaccCover}
Suppose \(\delta\) is a uniform inaccessible cardinal. Suppose \(i :V \to M \) is a \({<}\delta\)-irreducible ultrapower. Then \(\textnormal{Ord}^\delta\cap M_\delta\subseteq M\).
\begin{proof}
Let \(\kappa = \textsc{crt}(j_\delta)\). Then by \cite{Frechet} Theorem 7.8, \(\kappa\) is \(\delta\)-strongly compact. In particular, every regular cardinal in the interval \([\kappa,\delta]\) is uniform, so by \cref{IrredThm}, \(M_\delta\) is closed under \(\lambda\)-sequences for all \(\lambda < \delta\). Let \((i_*,j') : (M_\delta,M)\to N\) be the canonical comparison of \((j_\delta,i)\). By \cref{DeltaFactor}, \(i_*\) factors as \(i'\circ j_D\) where \(D\in \Un_{<\delta}\) and \(\textsc{crt}(i') > \delta\). Since \(V_\delta\subseteq M\), \(D\in M\). Therefore since \(M\) is closed under \(\textsc{sp}(D)\)-sequences, \(j_D\restriction M = j_D^M\) is amenable to \(M\). Therefore by \cref{CanonicalInternal}, \(j_D\restriction N\) is an internal ultrapower embedding of \(N\). Since \(N\subseteq M_D^{M_\delta}\), \(j_D\restriction\text{Ord}\) is amenable to \(M_D^{M_{\delta}}\), and it follows that \(D\) is principal. Hence \(\textsc{crt}(i_*) \geq \textsc{crt}(i') > \delta\). Thus \(\text{Ord}^\delta\cap M_\delta \subseteq N\subseteq M\), as desired.
\end{proof}
\end{prp}
\begin{cor}[UA]
Suppose \(\delta\) is a uniform inaccessible cardinal. The following are equivalent:
\begin{enumerate}[(1)]
\item \(M_\delta\) is closed under \(\delta\)-sequences.
\item Every \({<}\delta\)-irreducible ultrapower is closed under \(\delta\)-sequences.\qed
\end{enumerate}
\end{cor}
There seems to be no clear way forward, and this suggests a number of open questions:
\begin{qst}
Suppose \(\delta\) is strongly inaccessible. Can there be a countably complete ultrafilter \(U\) such that \(M = M_U\) has the following properties:
\begin{enumerate}[(1)]
\item \(M\) is closed under \(\lambda\)-sequences for all \(\lambda < \delta\).
\item \(M\) has the tight covering property at \(\delta\).
\item \(M\) is not closed under \(\delta\)-sequences.
\end{enumerate}
\end{qst}
Even assuming only ZFC, it is not clear that it is possible for such an ultrapower to exist. Regarding supercompactness at inaccessible cardinals, we do have the following intriguing fact:
\begin{thm}[UA]
Suppose \(\delta\) is a regular cardinal that carries distinct countably complete weakly normal ultrafilters. Then some \(\kappa < \delta\) is \(\delta\)-supercompact.\qed
\end{thm}
We omit the proof, but this raises another question:
\begin{qst}[UA]
Suppose \(\delta\) is a regular cardinal and \(\nu < \delta\) is a successor cardinal. Suppose \(\delta\) is a regular cardinal that carries distinct \(\nu\)-complete weakly normal ultrafilters. Is there a \(\delta\)-supercompact cardinal \(\kappa\) such that \(\nu < \kappa \leq \delta\)?
\end{qst}

Another interesting question is whether the ideas from the previous section suffice to characterize supercompactness at inaccessible cardinals:
\begin{conj}[UA]
Suppose \(\delta\) is an inaccessible cardinal. Suppose there is a \(\kappa\)-complete ultrafilter extending the club filter on \(P_\kappa(\delta)\). Then \(\kappa\) is either \(\delta\)-supercompact or a measurable limit of \(\delta\)-supercompact cardinals.
\end{conj}

\subsection{Almost huge cardinals}
\begin{thm}[UA]\label{Huge}
Suppose there is a countably complete weakly normal ultrafilter on a regular cardinal that concentrates on a fixed cofinality. Then there is an almost huge cardinal.
\end{thm}
Thus the same issues from \cref{Anom} prevent us from showing that there is a huge cardinal under these hypotheses.
We need a lemma which is useful in conjunction with \cref{IrredThm}.

\begin{defn}
Suppose \(\nu\) is a cardinal, \(\lambda\geq \nu\) is a regular cardinal, \(\mathcal F\) is a normal fine filter on \(P(\lambda)\). Then \(U^\nu_\delta(\mathcal F)\) is the \(\sE\)-least countably complete weakly normal ultrafilter isomoprhic to an extension of \(\mathcal F\) if it exists. 
\end{defn}

The proof of \cref{IrredExt} yields the following fact:

\begin{lma}[UA]\label{FIrred}
Suppose \(\nu\) is a successor cardinal, \(\lambda\geq \nu\) is a cardinal, and \(\mathcal F\) is a normal fine filter on \(P(\lambda)\). Then \(U^\nu_\delta(\mathcal F)\) is irreducible.\qed
\end{lma}

We just need the following corollary:
\begin{cor}[UA]\label{AIrred}
If \(\delta\) is a regular cardinal and \(A\subseteq \delta\), then the \(\sE\)-least weakly normal ultrafilter concentrating on \(A\) is irreducible.
\begin{proof}
Let \(\mathcal F\) be the club filter on \(\delta\) restricted to \(A\) viewed as a filter on \(P(\delta)\). Then \(\mathcal F\) is a normal fine filter on \(P(\delta)\). A weakly normal ultrafilter on \(\delta\) concentrates on \(A\) if and only if it is isomorphic to an extension of \(\mathcal F\). (For the forwards implication one needs that \(\delta\) is regular: weakly normal ultrafilters on singular cardinals need not extend the club filter.) Thus the \(\sE\)-least weakly normal ultrafilter concentrating on \(A\) is \(U^{\omega_1}_\delta(\mathcal F)\), which is irreducible by \cref{FIrred}.
\end{proof}
\end{cor}

\begin{proof}[Proof of \cref{Huge}]
Suppose \(\kappa < \delta\) are regular cardinals and \(U\) is the \(\sE\)-least countably complete weakly normal ultrafilter such that \(S^\delta_\kappa\in U\). By \cref{AIrred}, \(U\) is an irreducible ultrafilter. 

If \(\delta\) is a successor cardinal then it follows immediately from \cref{IrredThm} that \(U\) is \(\delta\)-supercompact.

Suppose \(\delta\) is a weakly inaccessible cardinal. By a theorem of Ketonen \cite{Ketonen}, \(U\) is \((\kappa^+,\delta)\)-regular, and therefore \(j_U\) is discontinuous at every regular cardinal in the interval \([\kappa^+,\delta]\). It now follows from \cref{IrredThm} that \(M_U\) is closed under \(\lambda\)-sequences for all \(\lambda < \delta\). Therefore some cardinal less than \(\delta\) is \(\delta\)-supercompact. By the main theorem of \cite{GCH} (Theorem 4.1), it follows that for all sufficiently large \(\lambda < \delta\), \(2^\lambda = \lambda^+\). In particular, \(\delta\) is strongly inaccessible.

By \cref{TightCover} and \cref{InaccCover}, it follows that \(M_U\) has the tight covering property at \(\delta\). Thus \(\text{cf}^{M_U}(\sup j_U[\delta]) = \delta\). On the other hand, \(j_U(\kappa) = \text{cf}^{M_U}(\sup j_U[\delta])\) since \(U\) is weakly normal and concentrates on \(S^\delta_\kappa\). Therefore \(j_U(\kappa) = \delta\). Thus \(j_U : V\to M\), \(j_U(\kappa) = \delta\), and \(M_U\) is closed under \(\lambda\)-sequences for all \(\lambda < \delta\). It follows that \(\textsc{crt}(U)\) is almost huge.
\end{proof}

\bibliography{Bibliography}{}
\bibliographystyle{unsrt}
\end{document}